\documentclass[english]{article}
\usepackage[T1]{fontenc}
\usepackage[latin9]{inputenc}
\usepackage{mathtools}
\usepackage{amsmath}
\usepackage{amsthm}
\usepackage{amssymb}
\usepackage{stackrel}
\usepackage{xcolor}
\usepackage[margin = 3 cm,marginpar= 2cm]{geometry}

\makeatletter
\newtheorem{defn}{Definition}[section]
\newtheorem{thm}{Theorem}[section]
\newtheorem{prop}{Proposition}[section]
\newtheorem{lem}{Lemma}[section]
\newtheorem{cor}{Corollary}[section]

\newtheorem{rem}{Remark}[section]
\newtheorem{claim}{Claim}[section]

\numberwithin{equation}{section}

\@ifundefined{date}{}{\date{}}
\makeatother

\usepackage{babel}

\DeclareMathOperator*{\sgn}{sgn}

\begin{document}
\title{Mean field limit for one dimensional opinion dynamics with Coulomb interaction and time dependent weights}

\author{
Immanuel Ben-Porat%
     \thanks{Mathematical Institute, University of Oxford, Oxford OX2 6GG, UK.  {Immanuel.BenPorat@maths.ox.ac.uk}} %
\and
    Jos\'e A. Carrillo%
     \thanks{Mathematical Institute, University of Oxford, Oxford OX2 6GG, UK.  {carrillo@maths.ox.ac.uk}} %
    \and 
    Sondre T. Galtung%
     \thanks{Department of Mathematical Sciences, Norwegian University of Science and Technology -- NTNU, 7491 Trondheim, Norway.{sondre.galtung@ntnu.no}}
}

\maketitle

\begin{abstract}
The mean field limit with time dependent weights for a 1D singular case, given by the attractive Coulomb interactions, is considered. This extends recent results \cite{1,8} for the case of regular interactions. The approach taken here is based on transferring the kinetic target equation to a Burgers-type equation through the distribution function of the measures. The analysis leading to the stability estimates of the latter equation makes use of Kruzkov entropy type estimates adapted to deal with nonlocal source terms.
\end{abstract}

\section{Introduction }

\subsection{General background }

In this paper, we are concerned with analyzing the mean field limit
of the following system of $2N$ ODEs:
\begin{equation}\label{eq:OPINION DYNAMICS}
\begin{cases}
\dot{x}_{i}^{N}(t) = \frac{1}{N}\stackrel[j=1]{N}{\sum}m_{j}^{N}(t)V'(x_{j}^{N}(t)-x_{i}^{N}(t)), & x_{i}^{N}(0)=x_{i}^{0,N} \\
\dot{m}_{i}^{N}(t) = \psi_{i}^{N}(\mathbf{x}_{N}(t),\mathbf{m}_{N}(t)), & m_{i}^{N}(0)=m_{i}^{0,N}.
\end{cases}
\end{equation}
The notation is as follows: the unknowns $x_{i}^N\in\mathbb{R}$ and
$m_{i}^N\in\mathbb{R}$ are referred to as the \textit{opinions} and
\textit{weights} respectively. The evolution of the opinions is given
in terms of the weights and a function $V:\mathbb{R}\rightarrow\mathbb{R}$
which is called \textit{the interaction} modulating the value of the opinion $x_i^N$ by the presence of the other opinions $x_j^N$. The evolution of the weights is given by means of functions $\psi_{i}^{N}:\mathbb{R}^{N}\times\mathbb{R}^{N}\rightarrow\mathbb{R}$
where we apply the notation 
\[
\mathbf{x}_{N}(t)\coloneqq(x_{1}^{N}(t),...,x_{N}^{N}(t)), \qquad \mathbf{m}_{N}(t)\coloneqq(m_{1}^{N}(t),...,m_{N}^{N}(t)).
\]
The weights $m_i^N(t)$ can be interpreted as the proportion of the total population with opinion $x_i^N(t)$.
How the system (\ref{eq:OPINION DYNAMICS}) originates from real-life
phenomena is beyond the scope of this work. Just to mention a few
works which explain how this system models phenomena in biology and
the social sciences, we refer to \cite{1},\cite{13} and \cite{16}.
The system (\ref{eq:OPINION DYNAMICS}) is a weighted version of the
first order $N$-body problem (simply by taking all the weights to
be identically equal to $1$), to which we now briefly draw our attention
to. As to give an extensive review of the relevant literature, we momentarily consider arbitrary $d\geq1$, although our main result is in 1D. By now, the mean field limit of the $N$-body problem 
\begin{equation}
\dot{x}_{i}^{N}(t)=\frac{1}{N}\stackrel[j=1]{N}{\sum}\nabla V(x_{j}^N(t)-x_{i}^N(t)), \qquad x_{i}^{N}(0)=x_{i}^{0,N}\label{N body problem}
\end{equation}
is fairly well understood, even in the case of interactions with strong
singularities near the origin---we will comment more about this later
on. This mean field limit is understood in terms of the empirical
measure which is defined by 
\[
\mu_{N}(t)\coloneqq\frac{1}{N}\stackrel[i=1]{N}{\sum}\delta_{x_{i}^N(t)}.
\]
Thanks to the work of Dobrushin \cite{7}, and assuming $\nabla V$ is Lipschitz, it is possible to show that $\mu_{N}(t)$ converges to the solution $\mu$ of the Vlasov equation 

\begin{equation}
\partial_{t}\mu(t,x)-\mathrm{div}(\mu\nabla V\star\mu)(t,x)=0, \qquad \mu(0,\cdot)=\mu^{0} \label{eq:homogenous transport}
\end{equation}
 with respect to the Wasserstein metric (provided this is true initially
of course). When time dependent weights are included the problem is
rendered difficult, since now the candidate for the empirical measure
is 
\[
\mu_{N}(t)\coloneqq\frac{1}{N}\stackrel[i=1]{N}{\sum}m_{i}^N(t)\delta_{x_{i}^N(t)}
\]
 and formal considerations (see Proposition 15 in \cite{7}) suggest the following transport equation with self-consistent source term as the target equation 

\begin{equation}
\partial_{t}\mu(t,x)-\mathrm{div}(\mu\nabla V\star\mu)(t,x)=h[\mu](t,x), \qquad \mu(0,\cdot)=\mu^{0}.\label{vlasov with source}
\end{equation}
Here $h[\mu]$ is the self-consistent source term which arises from
the inclusion of weights, and for the moment we do not specify it.
Already at the level of the well
posedness theory of the target equation (\ref{vlasov with source})
and the system (\ref{eq:OPINION DYNAMICS}) some care is needed, since
apriori it is not entirely that the solution stays a probability density
for all times. The well-posedness of equations (\ref{eq:OPINION DYNAMICS})
and (\ref{vlasov with source}) as well as the weighted mean field
limit has been successfully established in \cite{1} and \cite{8}
in arbitrary dimension and for interactions with Lipschitz gradient.
The approach in \cite{8} is based on stability estimates for the
Wasserstein distance, which imply both well-posedness and mean field
limit, whereas the approach in \cite{1} recovers to some extent the mean field limit obtained in \cite{8}, and is based on the graph limit
method. We refer also to \cite{14} and \cite{15} for more details
about the graph limit regime and its link with the mean field limit.
Other works which consider weighted opinion dynamics are \cite{10},
in which the weights are taken to be time-independent, but may vary
from one opinion to another, and \cite{16} which serves as a general
survey. 

\subsection{Main results }

All the existing literature reviewed so far concerns arbitrary dimension
and relatively well behaved potentials in terms of regularity, typically
with at least locally Lipschitz gradient. It is the aim of this work
to investigate how to overcome the challenges created due to singular
potentials in 1D. In particular, we consider the attractive 1D Coulomb
interaction $V(x)=\left|x\right|$ (so that $\partial_{x}V=\mathrm{sgn}(x)$,
with the convention $\mathrm{sgn}(0)=0$). In addition, we add further
limitations on the equation governing the weights, namely we take
\begin{equation}
\psi_{i}^{N}(\mathbf{x}_{N},\mathbf{m}_{N})=\frac{1}{N}\stackrel[j=1]{N}{\sum}m_{i}^N m_{j}^N S(x_{j}^N-x_{i}^N)\label{eq:special form}
\end{equation}
where $S\in C_{0}^{\infty}(\mathbb{R})$ is assumed to be odd. From the opinion modelling point of view, the value $|S(x_{j}^N-x_{i}^N)|$ in \eqref{eq:special form} can be interpreted as the rate of change from opinion $x_i^N$ to $x_j^N$ at a given time. In view of this interpretation the oddness of the function $S$ is natural from the assumption of conservation of the total population of individuals. In fact, it means that the proportion of individuals that change their opinion from value $x_i^N$ to value $x_j^N$ is the opposite of the individuals that change their opinion from value $x_j^N$ to value $x_i^N$ at a given time. Note that at a first glance it seems we allow that at any given time the individuals can change abruptly from opinion $x_i^N$ to opinion $x_j^N$. However, if the rate function $S$ is of compact support, as in our main result in Theorem \ref{Main Theorem }, then the allowed change of opinion is local and smooth. As we already mentioned, this
parity condition turns out to be important for the purpose of guaranteeing
preservation of the total mass. For these choices, the system (\ref{opinion dynamics})
takes the form 
\begin{equation}\label{opinion dynamics with kernel S}
\begin{cases}
\dot{x}_{i}^{N}(t) = \frac{1}{N}\stackrel[j=1]{N}{\sum}m_{j}^{N}(t)\sgn(x_{j}^{N}(t)-x_{i}^{N}(t)), & x_{i}^{N}(0)=x_{i}^{0,N}, \\
\dot{m}_{i}^{N}(t) = \frac{1}{N}\stackrel[j=1]{N}{\sum}m_{i}^{N}(t)m_{j}^{N}(t)S(x_{j}^{N}(t)-x_{i}^{N}(t)), & m_{i}^{N}(0)=m_{i}^{0,N}.
\end{cases}
\end{equation}
In this case, the mean field equation takes the form 
\[
\partial_{t}\mu(t,x)-\partial_{x}\left(\mu(t,x)
\int^x_{-\infty}\mu(t,y)dy-\mu(t,x)\int^\infty_{x}\mu(t,y)dy\right)=\mu(t,x)S\star\mu(t,x), \qquad
\mu(0,x)=\mu^{0}.
\]
When the kernel $S$ is odd one expects that $\mu$ stays a probability
density for all times and therefore the equation formally transforms
to
\begin{equation}\label{eq:-2}
\partial_{t}\mu(t,x)-\partial_{x}\left(\mu(t,x)\left(2\int^x_{-\infty}\mu(t,y)dy-1\right)\right)=\mu(t,x)S\star\mu(t,x), \qquad
\mu(0,x)=\mu^{0}.
\end{equation}
Setting $F(t,x)=-\frac{1}{2}+\int^x_{-\infty}\mu(t,y)dy$
and integrating both sides of equation (\ref{eq:-2}), we arrive (still
at the formal level) at the following non local Burgers type equation
for $F$

\begin{equation}
\partial_{t}F+\partial_{x}(A(F))=\mathbf{S}[F](t,x)\label{eq:-3}
\end{equation}
where 
$
A(F)\coloneqq-F^{2} 
$ and
\[
\mathbf{S}[F](t,x)\coloneqq F(t,x)(\phi\star F)(t,x)-\int^x_{-\infty}F(t,z)(\partial_{z}\phi\star F)(t,z)dz,\ \phi\coloneqq\partial_{x}S.
\]
Of course, the definition of $\mathbf{S}[F]$ is motivated by integrating
by parts the expression $\int^x_{-\infty}\partial_{z}F(t,z)(S\star\partial_{z}F)(t,x)dz$.
The idea of analyzing the equation for the primitive of $\mu$ stems
from the work \cite{4} which studies the homogenous Burgers equation.
The advantage of the equation (\ref{eq:-3}) in comparision to (\ref{eq:-2})
is that the flux term is local, which raises the hope that a well
posedness theory is in reach. We refer to \cite{6} for a well posedness
theory in the case of non-local fluxes (yet with no source term).
By closely adapting the method introduced in \cite{11} (which has
its roots in the classical Kruzkov entropy) we are able to simultaneously
prove the $BV$-well posedness of the equation (\ref{eq:-3}), and
derive the unique solution of (\ref{eq:-3}) in the limit as $N\rightarrow\infty$
of the primitive of the empirical measure (shifted by $-\frac{1}{2}$
in order to ensure proper cancellation), namely 
$$
F_{N}(t,x)\coloneqq-\frac{1}{2}+\frac{1}{N}\stackrel[k=1]{N}{\sum}m_{k}^{N}(t)H(x-x_{k}^{N}(t))
$$
where $H$ is the Heaviside function. A notable difference with respect to \cite{11} is that we deal with time dependent fluxes for the approximation sequence of conservation laws with sources satisfied by $F_N$ associated to the particle system \eqref{opinion dynamics with kernel S}, see \eqref{discretized Burgers} below. In addition, we also face technical difficulties in checking the entropy conditions of the accumulation limits of the sequence $F_N$ due to the time dependency. Other differences
which put our work in variance with \cite{11} are reflected in the
well-posedness for the system (\ref{eq:OPINION DYNAMICS}), the handling
of the different source term, and the verification that $F_{N}$ is an entropy
solution of (a discretized version of) equation (\ref{eq:-3}).
We stress that the 1D settings are essential for the analysis to be
carried out properly. In the context of higher dimensions, we mention
the recent breakthrough \cite{18}, in which the author succeeded
in deriving the equation (\ref{eq:homogenous transport}) from the
system (\ref{N body problem}) for interaction which may have even
stronger singularities than Coulomb. The modulated energy strategy
in \cite{18} is well suited for higher dimensions, which raises the
interesting question whether this strategy can be extended to the
framework of the present work in which time dependent weights are
considered. 

In Section 2 we fix some notation, prove the well posedness
of the system (\ref{eq:OPINION DYNAMICS}) for the case of the 1D
attractive Coulomb interaction and record other basic properties of the solutions
of the ODE system. Section 3 is aimed at constructing a discretized flux and showing that $
F_{N}
$ satisfies the Burgers equation associated with this flux. Section 4 is devoted towards proving Theorem \ref{Main Theorem }, the main result of this work. The mean field limit and
the $BV$-well posedness of equation (\ref{eq:-3}) are the content of Theorem \ref{Main Theorem }. 

\section{Preliminaries }

For readability, the upper index $N$ of the opinions/weights appears implicitly in the sequel. Consider the weighted $N$-body problem 
\begin{equation}\label{opinion dynamics}
\begin{cases}
\dot{x}_{i}(t)=\frac{1}{N}\underset{1\leq j\leq N:j\neq i}{\sum}m_{j}(t)\mathrm{sgn}(x_{j}(t)-x_{i}(t)), & x_{i}(0)=x_{i}^{0}, \\
\dot{m}_{i}(t) = \psi_{i}^{N}(\mathbf{x}_{N}(t),\mathbf{m}_{N}(t)), & m_{i}(0)=m_{i}^{0}.
\end{cases}
\end{equation}
For the purpose of establishing well posedness of the system (\ref{opinion dynamics}) on short times it is not strictly necessary to take $\psi_{i}^{N}$ to be of the
form (\ref{eq:special form}). However, in order to have well-posedness for all times this special form, and specifically the oddness of $S$ seems to be needed at least to some extent. The hypotheses that we impose on $\psi_{i}^{N}$
in the sequel are identical to those specified in \cite{1}
and are recalled below. 

\medskip{}
\textbf{Hypothesis (H1). }There is a constant $L>0$ such that for
all $\left(\mathbf{x}_{N},\mathbf{y}_{N},\boldsymbol{\mathbf{m}}_{N},\mathbf{p}_{N}\right)\in\mathbb{R}^{4N}$
the inequalities 

\[
\left|\psi_{i}^{N}(\mathbf{x}_{N},\boldsymbol{\mathbf{m}}_{N})-\psi_{i}^{N}(\mathbf{y}_{N},\boldsymbol{\mathbf{m}}_{N})\right|\leq L\left|\mathbf{x}_{N}-\mathbf{y}_{N}\right|,
\]

\[
\left|\psi_{i}^{N}(\mathbf{x}_{N},\boldsymbol{\mathbf{m}}_{N})-\psi_{i}^{N}(\mathbf{x}_{N},\boldsymbol{\mathbf{p}}_{N})\right|\leq L\left|\boldsymbol{\mathbf{m}}_{N}-\mathbf{p}_{N}\right|
\]
hold. In addition, there is a constant $C>0$ such that and for each
$\left(\mathbf{x}_{N},\boldsymbol{\mathbf{m}}_{N}\right)\in\mathbb{R}^{4N}$
it holds that 

\begin{equation}
\left|\psi_{i}^{N}(\mathbf{x}_{N},\boldsymbol{\mathbf{m}}_{N})\right|\leq C(1+\underset{1\leq k\leq N}{\max}\left|m_{k}\right|).\label{eq:sublinear functional}
\end{equation}
We shall hereafter supress the dependence on $N$ of the opinions/weights
whenever this dependence is irrelevant. Since we wish to prove well
posedness on a time interval which does not shrink to zero as $N\rightarrow\infty$,
a natural assumption to impose is that the opinions are seperated
initially, i.e. 
\begin{equation}
\forall N\in\mathbb{N}: x_{1}^{0}<...<x_{N}^{0}.\label{eq:seperation}
\end{equation}
\subsection{Fillipov solutions }

Before embarking on the task of establishing the stability estimates,
we must first prove the existence of a solution to the Equation (\ref{opinion dynamics with kernel S})
on some time interval $[0,T]$, where \textit{$T$ is independent of
$N$} (because eventually we wish to take $N\rightarrow\infty$).
It is the aim of this section to explain how this is achieved- in fact we will prove existence for arbitrarily large times. First,
let us be lucid about the notion of solution that we work with. We start by examining the more general ODE \ref{opinion dynamics} and we specify later  the stage in which the special form \ref{eq:special form} comes into play.
\begin{defn}
\label{strong solution } A \textit{classical solution} of the system
(\ref{opinion dynamics}) on $[0,T${]} is an absolutely continuous
curve $t\mapsto(\mathbf{x}_{N}(t),\boldsymbol{\mathbf{m}}_{N}(t))\in\mathrm{AC}([0,T];\mathbb{R}^{2N})$
such that the system (\ref{opinion dynamics}) is satisfied for a.e.
$t\in[0,T]$. Equivalently, for all $t\in[0,T]$ it holds that 
\begin{align*}
x_{i}(t) &= x_{i}^{0}+\int^t_{0}\frac{1}{N}\underset{1\leq j\leq N:j\neq i}{\sum}m_{j}(\tau)\sgn\left(x_{j}(\tau)-x_{i}(\tau)\right)d\tau, \\
m_{i}(t) &= m_{i}^{0}+\int^t_{0}\psi_{i}^{N}(\mathbf{x}_{N}(\tau),\mathbf{m}_{N}(\tau))d\tau .
\end{align*}
\end{defn}
In order to prove existence of solutions in the sense of Definition
\ref{strong solution } it is convenient to use the machinery of differential
inclusions as developed by Fillipov \cite{9}. The special form (\ref{opinion dynamics with kernel S})
is irrelevant for what concerns the abstract theory of Fillipov. First
we must review some basic definitions and facts from convex analysis.
Let us recall how to view the system (\ref{opinion dynamics}) as
a differential inclusion. Let $\Omega\subset\mathbb{R}^{l}$ be some
domain. Given a vector field $f:\Omega\rightarrow\mathbb{R}^{l}$
whose set of discontinuities is $D$ we assign to it a set valued map
$\mathcal{S}[f]:\Omega\rightarrow2^{\mathbb{R}^{l}}$ defined as 

\[
\mathcal{S}[f](X)\coloneqq\left\{ \underset{X^{\ast}\notin D}{\underset{X^{\ast}\rightarrow X}{\lim}}f(X^{\ast})\right\} .
\]
With this notation define the set valued map $\mathcal{F}[f]:\Omega\rightarrow2^{\mathbb{R}^{l}}$by
\begin{equation}
\mathcal{F}[f](X)\coloneqq\overline{\mathrm{co}}\mathcal{S}[f](X),\label{eq:definition of associated set valued map}
\end{equation}
where $\mathrm{co}$ stands for the convex hull. 
\begin{rem}
In the context of system (\ref{opinion dynamics}) we take $f(\mathbf{x}_{N},\mathbf{m}_{N})$ to be the vector
\begin{equation}
\left(\frac{1}{N}\underset{1\leq j\leq N}{\sum}m_{1}\mathrm{sgn}(x_{j}-x_{1}),\ldots,\frac{1}{N}\underset{1\leq j\leq N}{\sum}m_{N}\mathrm{sgn}(x_{j}-x_{N}),\psi_{1}^{N}(\mathbf{x}_{N},\mathbf{m}_{N}),\ldots,\psi_{N}^{N}(\mathbf{x}_{N},\mathbf{m}_{N})\right).\label{explicit formula for vector field}
\end{equation}
\end{rem}
Next, we recall the notion of upper-semicontinuity. 
\begin{defn}
For each closed sets $A,B\subset\Omega$ let $\beta(A,B)\coloneqq \sup_{a\in A} \inf_{b\in B} \left|a-b\right|$.
A function $\mathcal{F}:\Omega\rightarrow2^{\mathbb{R}^{l}}$ is said
to be \textit{upper-semicontinuous at $p\in\Omega$} if $\underset{p'\rightarrow p}{\lim}\beta(\mathcal{F}(p),\mathcal{F}(p'))=0$.
If $\mathcal{F}$ is upper-semicontinuous for each $p\in\Omega$ then
we say that it is \textit{upper-semicontinuous. }
\end{defn}
On the other hand, recall the notion of a solution to a differential
inclusion 
\begin{defn}
\label{definition of solution to differential inclusion } Let $\mathcal{F}:\Omega\rightarrow2^{\mathbb{R}^{l}}$
be a set valued map. A \textit{solution }on $[\underline{T},\overline{T}]$
to the differential inclusion 
\begin{equation}
\dot{X}(t)\in\mathcal{F}(X(t)) \label{eq:dincl}
\end{equation}
 is an absolutely continuous map $X:[\underline{T},\overline{T}]\rightarrow\mathbb{R}^{l}$
such that \eqref{eq:dincl} holds
for a.e. $t\in[\underline{T},\overline{T}]$. 
\end{defn}
The existence theory of differential inclusions has been developed
by several different groups (Just to mention a few, see \cite{2},\cite{3}
and \cite{9}). We follow \cite{9}, and by no means aim to give an
exhaustive overview of this rich theory. The following theorem is
the main tool that we need to get existence for the opinion dynamics. 
\begin{thm}
\textup{\label{existence of filippov solution }(\cite{9}, Theorem
1, Page 77) }Let $\mathcal{F}:\mathbb{R}^{l}\rightarrow2^{\mathbb{R}^{l}}$
be such that 

1. $\mathcal{F}(x)\neq\emptyset$ for all $x\in\Omega$. 

2. $\mathcal{F}(x)$ is compact and convex for all $x\in\Omega$.

3. $\mathcal{F}$ is upper-semicontinuous. 

Then for any $\overline{T}>0$ and any $X_{\underline{T}}\in\mathbb{R}^{l}$
there exist a solution (in the sense of Definition \ref{definition of solution to differential inclusion })
to the differential inclusion 

\[
\dot{X}(t)\in\mathcal{F}(X(t)),\ X(\underline{T})=X_{\underline{T}}
\]
on the time interval $[\underline{T},\overline{T}]$.
\end{thm}
Here is a convenient criteria for upper semi-continuity 
\begin{lem}
\textup{(\cite{9}, Lemma 3, Page 67)} \label{set valued map is upper semicontinuous }
Let $f:\Omega\rightarrow\mathbb{R}^{l}$ be piecewise continuous and
let $\mathcal{F}[f]$ be defined as in (\ref{eq:definition of associated set valued map}).
Then $\mathcal{F}[f]$ is upper-semicontinuous. 
\end{lem}
We can now gather all of the above reminders to get the following
important conclusion 
\begin{cor}
\label{Corollary } Let $\psi_{i}^{N}$ satisfy (H1). Let $f:\mathbb{R}^{2N}\rightarrow\mathbb{R}^{2N}$
be given by (\ref{explicit formula for vector field}). For any $\overline{T}>0$
and any $(\mathbf{x}_{N}^{\underline{T}},\mathbf{m}_{N}^{\underline{T}})\in\mathbb{R}^{2N}$
there exist a solution to the differential inclusion 
\[
(\dot{\mathbf{x}}_{N},\dot{\mathbf{m}}_{N})\in\mathcal{F}[f](\mathbf{x}_{N}(t),\mathbf{m}_{N}(t)), \qquad (\mathbf{x}_{N}(\underline{T}), \mathbf{m}_{N}(\underline{T}))=(\mathbf{x}_{N}^{\underline{T}},\mathbf{m}_{N}^{\underline{T}})
\]
on $[\underline{T},\overline{T}]$. 
\end{cor}
\begin{proof}
That $\mathcal{S}[f](p)$ is compact is an immediate consequence of
the easy observation that $\mathcal{S}[f]$ is finite set-valued map.
For the same reason $\mathcal{F}[f](p)$ is the set of all convex
combinations of the points of $\mathcal{S}[f](p),$ and is therefore
compact (see, e.g., page 62 in \cite{9}). By Lemma \ref{set valued map is upper semicontinuous },
$\mathcal{F}[f]$ is upper-semicontinuous, so that by Theorem \ref{existence of filippov solution }
the claim follows. 
\end{proof}
It is important to observe that when $f$ is continuous, the map $\mathcal{F}[f]$
is single valued and as a result the corresponding differential inclusion
reduces to an ODE in the sense of Definition \ref{strong solution }.
This observation, as well as the special structure which arises from
the attractive interaction and forces opinions to stick together,
allows us to construct solutions to the system (\ref{opinion dynamics})
in the sense of Definition \ref{strong solution }. We now return to the form \ref{opinion dynamics with kernel S}. The following proposition is the main result of this section. 
\begin{prop}
\label{Existence } Let assumption (\ref{eq:seperation})
hold and suppose $m_{i}^{\mathrm{0}}>0$ and $\frac{1}{N}\stackrel[i=1]{N}{\sum}m_{i}^{\mathrm{0}}=1$. Let $S$ be odd, continuous and bounded. For each $\overline{T}>0$  there exist a solution on $[0,\overline{T}]$
to the system (\ref{opinion dynamics with kernel S}) (in the sense
of Definition (\ref{strong solution })). 
\end{prop}
\textit{Proof}\textbf{. Step 1(a).} \textbf{Preservation of total
mass. }Let $\overline{T}>0$ and let $(\mathbf{x}_{N}(t),\mathbf{m}_{N}(t))$
be some solution on $[0,\overline{T}]$ given by Corollary \ref{Corollary }.
For each $1\leq i\leq2N$ let $\mathcal{\mathfrak{f}}_{i}:\mathbb{R}^{2N}\rightarrow\mathbb{R}^{2N}$
be the vector field whose $i$-th component identifies with the $i$-th
component of $f$ and is $0$ in all other components, so that $f=\stackrel[i=1]{2N}{\sum}\mathcal{\mathfrak{f}}_{i}$.
By subadditivity of $\mathcal{F}[f]$ we have 

\[
\mathcal{F}[f]\subset\stackrel[i=1]{2N}{\cup}\mathcal{F}[\mathcal{\mathfrak{f}}_{i}].
\]
Since the functions $\psi_{i}^{N}$ are continuous, $\mathcal{F}[\mathcal{\mathfrak{f}}_{i}]=\left\{ \mathcal{\mathfrak{f}}_{i}\right\} $
for each $N+1\leq i\leq2N$, so that the weights are governed by an
ODE, namely for a.e. $t\in[0,\overline{T}]$ one has 

\[
\dot{m}_{i}(t)=\psi_{i}^{N}(\mathbf{x}_{N}(t),\mathbf{m}_{N}(t)), \qquad m_{i}(0)=m_{i}^{0}.
\]
We claim that the total mass is preserved. Indeed, using that $S$ is odd, let us compute 
\begin{align*}
\frac{1}{N}\stackrel[k=1]{N}{\sum}\dot{m}_{k}(t) &= \frac{1}{N^{2}}\stackrel[k=1]{N}{\sum}\stackrel[j=1]{N}{\sum}m_{k}(t)m_{j}(t)S(x_{j}(t)-x_{k}(t)) =-\frac{1}{N^{2}}\stackrel[k=1]{N}{\sum}\stackrel[j=1]{N}{\sum}m_{k}(t)m_{j}(t)S(x_{k}(t)-x_{j}(t)) \\
&=-\frac{1}{N^{2}}\stackrel[j=1]{N}{\sum}m_{j}(t)\stackrel[k=1]{N}{\sum}m_{k}(t)S(x_{k}(t)-x_{j}(t)) =-\frac{1}{N}\stackrel[j=1]{N}{\sum}\dot{m}_{j}(t).
\end{align*}
As a result we get 
\[
\frac{d}{dt}\frac{1}{N}\stackrel[k=1]{N}{\sum}m_{k}(t)=0,
\]
so that 
\[
\frac{1}{N}\stackrel[k=1]{N}{\sum}m_{k}(t)=\frac{1}{N}\stackrel[k=1]{N}{\sum}m_{k}^{\mathrm{0}}=1.
\]

\textbf{Step 1(b).Weights remain positive}. We claim that $m_{i}(t)>0$ for all $t\in[0,\overline{T}]$.
Indeed using preservation of total mass we obtain

\[
\left|\frac{d}{dt}\frac{1}{2}\log\left(m_{i}^{2}+\varepsilon^{2}\right)\right|=\left|\frac{m_{i}(t)\dot{m_{i}}(t)}{m_{i}^{2}+\varepsilon^{2}}\right|=\left|\frac{m_{i}^{2}(t)}{(m_{i}^{2}(t)+\varepsilon^{2})N}\stackrel[j=1]{N}{\sum}m_{j}(t)S(x_{j}(t)-x_{i}(t))\right|\leq\left\Vert S\right\Vert _{\infty},
\]
hence 

\[
-\left\Vert S\right\Vert _{\infty}\leq\frac{1}{2}\frac{d}{dt}\log\left(m_{i}^{2}(t)+\varepsilon^{2}\right)\leq\left\Vert S\right\Vert _{\infty}.
\]
Integration in time yields the for all $t\in[0,\overline{T}]$ the
estimate

\[
-2\left\Vert S\right\Vert _{\infty}t+\log\left((m_{i}^{\mathrm{0}})^{2}+\varepsilon^{2}\right)\leq\log\left(m_{i}^{2}+\varepsilon^{2}\right)\leq 2\left\Vert S\right\Vert _{\infty}t+\log\left((m_{i}^{\mathrm{0}})^{2}+\varepsilon^{2}\right),
\]
and consequently 
\[
\left((m_{i}^{\mathrm{0}})^{2}+\varepsilon^{2}\right)\exp(-2\left\Vert S\right\Vert _{\infty}t)\leq m_{i}^{2}(t)+\varepsilon^{2}\leq\left((m_{i}^{\mathrm{0}})^{2}+\varepsilon^{2}\right)\exp(2\left\Vert S\right\Vert _{\infty}t).
\]
Letting $\varepsilon\rightarrow0$ gives 
\[
(m_{i}^{\mathrm{0}})^{2}\exp(-2\left\Vert S\right\Vert _{\infty}t)\leq m_{i}^{2}(t)\leq(m_{i}^{\mathrm{0}})^{2}\exp(2\left\Vert S\right\Vert _{\infty}t).
\]
In particular $m_{i}(t)$ does not vanish on $[0,\overline{T}]$,
which by continuity and the assumption $m_{i}^{\mathrm{in}}>0$ implies
$m_{i}(t)>0$ for all $t\in[0,\overline{T}]$ with the estimate 

\[
m_{i}^{\mathrm{0}}\exp(-\left\Vert S\right\Vert _{\infty}t)\leq m_{i}(t)\leq m_{i}^{\mathrm{0}}\exp(\left\Vert S\right\Vert _{\infty}t).
\]

\textbf{Step }2. \textbf{Construction of a classical solution on $[0,\overline{T}]$}.
The idea is to apply an iteration argument which terminates after
at most $N$ steps. This iteration proceeds as follows. The assumption
that $\forall i\neq j:x_{i}^{0}\neq x_{j}^{0}$ implies that $\forall i\neq j:x_{i}(t)\neq x_{j}(t)$
on some sufficiently (possibly $N$-dependent) short time interval
(by continuity). As a result the solution is classical on some short
time- by the same argument at the beginning of Step 1(a). Let 

\[
T_{1}^{\ast}=\sup\left\{ 0<T\leq\overline{T}\left|\forall t\in[0,T),\forall i\neq j:x_{i}(t)\neq x_{j}(t)\right.\right\} .
\]
If $T_{1}^{\ast}=\overline{T}$ we are done. If $T_{1}^{\ast}<\overline{T}$,
consider the set of all collisions at time $T_{1}^{\ast}$ with the
$i$-th particle, i.e. 

\[
J_{1}^{i}\coloneqq\left\{ j\neq i\left|x_{i}(T_{1}^{\ast})=x_{j}(T_{1}^{\ast})\right.\right\} ,
\]
and set

\[
f_{i}^{1}(\mathbf{x}_{N},\mathbf{m}_{N})\coloneqq\frac{1}{N}\underset{j\notin J_{1}^{i}}{\sum}m_{j}\mathrm{sgn}(x_{j}-x_{i}),
\]
and 
\[
f^{1}\coloneqq(f_{1}^{1},...,f_{N}^{1}).
\]
Obviously $J_{1}^{i}=J_{1}^{j}$ iff $j\in J_{1}^{i}$. From the same
consideration as before the solution of the differential inclusion
\[
\dot{\mathbf{x}}_{N,1}\in\mathcal{F}[f^{1}](\mathbf{x}_{N,1}(t)),\ \mathbf{x}_{N,1}(T_{1}^{\ast})=\mathbf{x}_{N}(T_{1}^{\ast})
\]
satisfies $x_{i}(t)\neq x_{j}(t)$ for all $1\leq i\leq N$ and $j\in\left[N\right]\setminus J_{1}^{i}$
with $i\neq j$, and is classical (both of these conclusions hold
on some short time of course). Let 
\[
T_{2}^{\ast}=\sup\left\{ T_{1}^{\ast}<T\leq\overline{T}\left|\forall t\in[T_{1}^{\ast},\overline{T}),\forall1\leq i\leq N,\forall j\in\left[N\right]\setminus J_{1}^{i},i\neq j:x_{i}(t)\neq x_{j}(t)\right.\right\} .
\]

If $T_{2}^{\ast}=\overline{T}$, then consider the curve 
\[
\mathbf{y}_{N}(t)=\left\{ \begin{array}{c}
\mathbf{x}_{N}(t),\ t\in[0,T_{1}^{\ast}]\\
\mathbf{x}_{N,1}(t)\ t\in(T_{1}^{\ast},\overline{T}]
\end{array}\right..
\]

In this case we make use of the attractive structure, which implies
that opinions which collided remain so, in order to show that $\mathbf{y}_{N}(t)$
is a solution:
\begin{claim}
\label{Particles stick together }The curve $(\mathbf{y}_{N}(t),\mathbf{m}_{N}(t))$
is a solution to the system (\ref{opinion dynamics with kernel S})
on $[0,\overline{T}]$. 
\end{claim}
\textit{Proof}. The curve $\mathbf{y}_{N}(t)$ solves the Equation
(\ref{opinion dynamics with kernel S}) on $[0,T_{1}^{\ast}]$. To
see that $\mathbf{y}_{N}(t)$ solves the equation on $(T_{1}^{\ast},\overline{T}]$
we need to explain why particles that collided at time $T_{1}^{\ast}$
remain collided for all $t>T_{1}^{\ast}$. Indeed, suppose on the
contrary that $x_{i_{0}}(T_{1}^{\ast})=x_{j_{0}}(T_{1}^{\ast})$ for
some $i_{0}\neq j_{0}$, but 
\[
\tau\coloneqq\mathrm{inf}\left\{ T_{1}^{\ast}<t\leq\overline{T}\left|x_{i_{0}}(t)\neq x_{j_{0}}(t)\right.\right\} >0,
\]
so that $x_{i_{0}}(\tau)=x_{j_{0}}(\tau)$. We may assume with no
loss of generality that $x_{i_{0}}(t)-x_{j_{0}}(t)>0$ on some sufficiently
small time interval $(\underline{\tau},\overline{\tau})\subset(\tau,\overline{T}]$.
Let 

\[
\tau_{\ast}\coloneqq\inf\left\{ \left.\tau<\tau'<\underline{\tau}\right|\forall t\in(\tau',\overline{\tau}):x_{i_{0}}(t)-x_{j_{0}}(t)>0\right\} ,
\]
so that $x_{i_{0}}(t)-x_{j_{0}}(t)>0$ for all $t\in(\tau_{\ast},\overline{\tau})$
and $x_{i_{0}}(\tau_{\ast})=x_{j_{0}}(\tau_{\ast})$. On the other
hand, since $J_{1}^{i_{0}}=J_{1}^{j_{0}}$, for a.e. $t\in(\tau_{\ast},\overline{\tau})$
we compute that 
\begin{equation}\label{eq:-5}
\begin{aligned}
\dot{x}_{i_{0}}(t)-\dot{x}_{j_{0}}(t) &= \frac{1}{N}\underset{j\notin J_{1}^{i_{0}}}{\sum}m_{j}(t)(\mathrm{sgn}(x_{j}(t)-x_{i_{0}}(t))-\mathrm{sgn}(x_{j}(t)-x_{j_{0}}(t))) \\
&= \frac{1}{N}\underset{j\notin J_{1}^{i_{0}}:x_{i_{0}}(t)\geq x_{j}\geq x_{j_{0}}(t)}{\sum}m_{j}(t)(\mathrm{sgn}(x_{j}-x_{i_{0}})-\mathrm{sgn}(x_{j}-x_{j_{0}}))\leq0,
\end{aligned}
\end{equation}
which implies that for all $t\in(\tau_{\ast},\overline{\tau})$

\[
x_{i_{0}}(t)-x_{j_{0}}(t)\leq x_{i_{0}}(\tau_{\ast})-x_{j_{0}}(\tau_{\ast})=0,
\]
which is absurd. Remark that Inequality (\ref{eq:-5}) is thanks to
the fact the the weights are positive (step 1(b)). Therefore $x_{i_{0}}(t)=x_{j_{0}}(t)$
for all $t>T_{1}^{\ast}$ which shows that for all $T_{1}^{\ast}<t\leq\overline{T}$
it holds that
\[
\frac{1}{N}\underset{j\notin J_{1}^{i_{0}}}{\sum}\mathrm{sgn}(x_{j}(t)-x_{i}(t))=\frac{1}{N}\stackrel[j=1]{N}{\sum}\mathrm{sgn}(x_{j}(t)-x_{i}(t)),
\]
so that $\mathbf{y}_{N}(t)$ is a solution to the system (\ref{eq:OPINION DYNAMICS})
on $(T_{1}^{\ast},\overline{T}]$ as well. 
\begin{flushright}
$\square$
\par\end{flushright}

If $T_{2}^{\ast}<\overline{T}$, then we continue according to the
algorithm described above, which must terminate after at most $N$
steps, thereby yielding an absolutely continuous curve $\mathbf{y}_{N}(t)$
on $[0,\overline{T}]$. The same argument demonstrated in Claim (\ref{Particles stick together })
shows that $\mathbf{y}_{N}(t)$ is a solution to the system (\ref{opinion dynamics with kernel S})
. 
\begin{flushright}
$\square$
\par\end{flushright}

We finish this section by observing a few elementary properties of
solutions to (\ref{opinion dynamics with kernel S}). First we stress
that \textit{any }classical solution on some given time interval to
the system (\ref{opinion dynamics with kernel S}) must satisfy
the ``sticky opinions property'', namely that opinions that collide
stay collided. As a result there exist a finite partition of the time
interval into sub-intervals on each of which no new collisions occur. Moreover, opinions preserve the initial ordering. This is summarized in the following 
\begin{prop}
\label{prop:} Let the assumptions of Proposition \ref{Existence } hold. Let $(\mathbf{x}_{N}(t),\mathbf{m}_{N}(t))$ be a solution
to system (\ref{opinion dynamics with kernel S}) (in the sense of
Definition (\ref{strong solution })) on some time interval $[0,\overline{T}]$.
There exist $0=T_{0}<T_{1}<...<T_{k-1}<T_{k}=\overline{T}$ ($1\leq k\leq N$) such that for any
given $1\leq l\leq k$ it holds that for any given $1\leq i,j\leq N$ either
$\forall t\in[T_{l-1},T_{l}):x_{i}(t)=x_{j}(t)$ or $\forall t\in[T_{l-1},T_{l}):x_{i}(t)\neq x_{j}(t)$. Moreover, $x_{1}(t)\leq...\leq x_{N}(t),\ t\in[0,\overline{T}]$.
\end{prop}
A careful examination of the argument in Proposition \ref{Existence }
reveals that the proof of Proposition \ref{prop:} is in fact implicitly
included in that of Proposition \ref{Existence }, and is therefore
omitted. This is just to clarify that the ''sticky opinions property''
is not some kind of an extra assumption. In addition, for the mean
field limit it is important to observe that 
\begin{equation}
F_{N}(t,x)\coloneqq-\frac{1}{2}+\frac{1}{N}\stackrel[k=1]{N}{\sum}m_{k}(t)H(x-x_{k}(t)),\label{eq:FN def section 2}
\end{equation}
becomes constant for any $x$ outside some interval, which is a consequence
of preservation of total mass- this is precisely the place where we
use the assumption that $S$ is odd. This would allow to obtain stability
estimates globally in $L^{1}$ (rather only in $L_{\mathrm{loc}}^{1}$)
and is recorded in the following simple 
\begin{lem}
\label{Long time constant } Let the assumptions of Proposition \ref{Existence } hold. Suppose also there is some $\overline{X}$
such that for all $N\in\mathbb{N}$ and $1\leq i\leq N$ it holds
that $\left|x_{i}^{\mathrm{0}}\right|\leq\overline{X}$. Then,
there exist some $\overline{R}=\overline{R}(\left\Vert S\right\Vert _{\infty},\overline{X},\overline{T})>0$ such
that for any $F_{N}(t,\pm x)=\pm\frac{1}{2}$ for all $t\in[0,\overline{T}]$
and $\left|x\right|>\overline{R}$.
\end{lem}
\begin{proof}
Thanks to the equation for the opinions, the bound for the masses
in step 1 of Proposition \ref{Existence } and the assumption on the
initial opinions, we have 

\[
\left|x_{i}(t)\right|\leq\left|x_{i}^{\mathrm{0}}\right|+\frac{1}{N}\underset{1\leq j\leq N:j\neq i}{\sum}\int^t_{0}\left|m_{j}(\tau)\right|d\tau\leq\overline{X}+\overline{M}\overline{T},
\]
for some $\overline{M}=\overline{M}(\left\Vert S\right\Vert _{\infty},\overline{T})$.
Set $\overline{R}=\overline{R}(\overline{T},\overline{X},\overline{M})\coloneqq\overline{X}+\overline{M}\overline{T}$.
Then according, for all $x\geq\overline{R}$ it holds that 

\[
H(x-x_{k}(t))=1
\]
while for all $x<-\overline{R}$ it holds that 

\[
H(x-x_{k}(t))=0.
\]
Keeping in mind step 1 of Proposition \ref{Existence } we conclude
that for all $x>R$ and $t\in[0,\overline{T}]$ 

\[
F_{N}(t,x)=-\frac{1}{2}+\frac{1}{N}\sum_{k=1}^N m_{k}(t)=-\frac{1}{2}+1=\frac{1}{2},
\]
while for all $x<-\overline{R}$ and $t\in[0,\overline{T}]$ 

\[
F_{N}(t,x)=-\frac{1}{2}+0=-\frac{1}{2}.
\]
\end{proof}

\section{\label{sec:The-Discretized-Version}The discretized version of the
Burgers-type equation }

In this section we construct a discretization of the flux by means
of the weights, and then show that $F_{N}$ is an entropy solution
for the corresponding Burgers like equation-- that is equation (\ref{eq:-3}),
with the only difference being that the flux $A$ is replaced by a
discretized approximation thereof. As usual, we assume that the opinion are
ordered increasingly, i.e., 

\[
x_{1}^{0}<x_{2}^{0}<\cdots<x_{N}^{0}.
\]
Throughout this section we always work under the assumptions of Proposition \ref{Existence } and take  $t\mapsto(\mathbf{x}_{N}(t),\boldsymbol{\mathbf{m}}_{N}(t))$ to be a solution of the system \ref{opinion dynamics with kernel S}. For each $0\leq i\leq N$ we set 

\[
\theta_{i}(t)\coloneqq\left\{ \begin{array}{c}
-\frac{1}{2}+\frac{1}{N}\stackrel[j=1]{i}{\sum}m_{j}(t),\ \ 1\leq i\leq N,\\
-\frac{1}{2},\ \ i=0.
\end{array}\right.
\]
Note that since the weights are positive for all times (step 1 in
Proposition \ref{Existence } ) the $\theta_{i}$ are ordered increasingly
\[
-\frac{1}{2}\equiv\theta_{0}(t)<\cdots<\theta_{N}(t)\equiv\frac{1}{2}.
\]
The discretized flux, denoted $A_{N}(t,x)$, is defined for each $t$
to be the (unique) continuous, piecewise linear function with break
points only at $\left(\theta_{i}(t)\right)_{i=1}^{N-1}$ such that
$A(t,\theta_{i}(t))=A(\theta_{i}(t))$. In other words for each $x\in(\theta_{i-1}(t),\theta_{i}(t))$

\[
A_{N}(t,x)\coloneqq\left(\frac{A(\theta_{i}(t))-A(\theta_{i-1}(t))}{\theta_{i}(t)-\theta_{i-1}(t)}\right)(x-\theta_{i}(t))+A(\theta_{i}(t)).
\]
Although the $A_{N}$ are time dependent, the following simple lemma
shows that they approximate the time independent flux $A$.
\begin{lem}
\label{AN approximates A} For each $t\in[0,T]$ and each $x\in[-\frac{1}{2},\frac{1}{2}]$
it holds that 

\[
\left|A_{N}(t,x)-A(x)\right|\leq\frac{2\overline{M}\mathrm{Lip}(\left.A\right|_{[-\frac{1}{2},\frac{1}{2}]})}{N}.
\]
where $\overline{M}$ is as in Lemma \ref{Long time constant }.
\end{lem}
\begin{proof}
Keep $t\in[0,T]$ fixed. Let $x\in[-\frac{1}{2},\frac{1}{2}]$, and
pick $1\leq i\leq N$ such that $x\in[\theta_{i-1}(t),\theta_{i}(t)]$.
By the above formula for $A_{N}$ we get 

\[
\left|A_{N}(t,x)-A(x)\right|\leq\left|\frac{A(\theta_{i}(t))-A(\theta_{i-1}(t))}{\theta_{i-1}(t)-\theta_{i}(t)}\left(x-\theta_{i}(t)\right)\right|+\left|A(\theta_{i}(t))-A(x)\right|.
\]
Clearly 
\[
\left|\frac{x-\theta_{i}(t)}{\theta_{i-1}(t)-\theta_{i}(t)}\right|\leq1,
\]
so that 
\begin{align*}
\left|A_{N}(t,x)-A(x)\right|&\leq\left|A(\theta_{i}(t))-A(\theta_{i-1}(t))\right|+\left|A(\theta_{i}(t))-A(x)\right|
\\
&\leq2\mathrm{Lip}(\left.A\right|_{[-\frac{1}{2},\frac{1}{2}]})\left|\theta_{i}(t)-\theta_{i-1}(t)\right|\leq\frac{2\overline{M}\mathrm{Lip}(\left.A\right|_{[-\frac{1}{2},\frac{1}{2}]})}{N}.
\end{align*}
\end{proof}
The definition of entropy solutions originated in the celebrated work
\cite{12}. We recall the notion of entropy solution for conservation
laws with non-local source terms, which will be the the notion of
solution that we will use for what concerns the mean field limit and
the well posedness for the Cauchy problem of the Burgers equation. 
\begin{defn}
\label{definition of kruzkov solution-1} Let $A\in C([0,T];\mathrm{Lip}[-\frac{1}{2},\frac{1}{2}])$.
A function $F\in BV([0,T]\times\mathbb{R})$ is called an entropy
solution of the equation 

\[
\partial_{t}F+\partial_{x}(A(t,F))=\mathbf{S}[F](t,x)
\]
with initial data $F^{\mathrm{0}}\in BV(\mathbb{R})$
if 

(1) $\forall t\in[0,T],\forall x\geq R:F(t,\pm x)=\pm\frac{1}{2}$.

(2) The map $x\mapsto F(t,x)$ is non-decreasing for any $t\in[0,T]$. 

(3) $F(t, \cdot)$ converges to $F^{\mathrm{0}}$ as $t\to 0^+$ in the sense of distributions.

(4) For any $\chi\in C_{0}^{\infty}((0,T)\times\mathbb{R})$ and any
$\alpha\in\mathbb{R}$ it holds that 

\begin{equation}
\int_0^T\int_\mathbb{R}\left(\left|F-\alpha\right|\chi_{t}+\mathrm{sgn}(F-\alpha)\left(A(F)-A(\alpha)\right)\partial_{x}\chi+\mathrm{sgn}(F-\alpha)\chi\mathbf{S}[F](t,x)\right)dxdt\geq0.\label{entropy inequality}
\end{equation}
\end{defn}
\begin{rem}
A particular byproduct of the requirements (1) and (2) is that for any
$t\in[0,T]$ the total variation $\left|\partial_{x}F(t,\cdot)\right|$
is a probability measure. 
\end{rem}
\begin{rem}
We emphasize that we can keep the classical definition of
the convolution for the term  $\mathbf{S}[F]$ unlike in \cite{11}. That is, if $\mathrm{supp}(S)\subset[-r,r]$
we take $R=\max\left\{ r,\overline{R}\right\} $($\overline{R}$ as
in Lemma \ref{Long time constant }) and realizing that 

\[
\phi\star F(t,z)\coloneqq \int_{-2R}^{2R} \phi(z-\zeta)F(t,\zeta)d\zeta,
\]
due to the definition of $S\star F$, $\phi=\partial_z S$, and the fact that $F$ is constant outside the interval $[-2R,2R]$.
It is now evident that $\phi\star F,\partial_{z}\phi\star F\in C_{0}([0,T]\times\mathbb{R})$,
which enables to make sense of the last term in the left hand side
of inequality (\ref{entropy inequality}). 
\end{rem}
Next we claim that $F_{N}$ is an entropy solution to our Burgers
equation but with a discretized flux,. i.e. the equation 

\begin{equation}
\partial_{t}F_{N}+\partial_{x}(A_{N}(t,F_{N}))=\mathbf{S}[F_{N}](t,x).\label{discretized Burgers}
\end{equation}
To verify this we need to check that

i. $F_{N}$ is a classical solution on finitely many regions which
form a disjoint partition of the whole domain $(0,T)\times\mathbb{R}$ 

ii. The Rankine-Hugoniot condition 

iii. The Oleinik conditions. 

It is classical that the verification of i-iii imply the the integral
inequality (\ref{entropy inequality})- see Appendix \ref{Appendix A}
for more details about this implication. We start by verifying point
1. Let us recall that by Proposition \ref{prop:} we know there exist
finitely many times $0=T_{0}<T_{1}<...<T_{k-1}<T_{k}=T$ such that
on each $(T_{j-1},T_{j})$ collision does not occur. More specifically,
we know that for each $1\leq j\leq k$ there is a disjoint partition
of $\left[N\right]$ into subsets $I_{1}^{j},...,I_{m_{j}}^{j}\subset\left[N\right]$
(in brief $\stackrel[i=1]{m_{j}}{\sqcup}I_{i}^{j}=\left[N\right]$)
such that: 

(a) Given $1\leq i\leq m_{j}$ it holds that $x_{\alpha}(t)=x_{\beta}(t)$
for each $\alpha,\beta\in I_{i}^{j}$ and each $t\in[T_{j-1},T_{j})$. 

(b) For each $1\leq i<i'\leq m_{j}$ it holds that $x_{\alpha}(t)<x_{\beta}(t)$
for each $\alpha\in I_{i}^{j},\beta\in I_{i'}^{j}$ and each $t\in[T_{j-1},T_{j})$. 

For each $1\leq j\leq k$ and $1\leq i\leq m_{j}$ we let $i^{\ast}(j)$
be the maximal index inside $I_{i}^{j}$, i.e. 

\[
i^{\ast}(j)\coloneqq\max\left\{ r\left|r\in I_{i}^{j}\right.\right\} .
\]
With this notation we have the following simple 

\begin{lem}
\label{FN is classical solution on a partition } Let $F_{N}$ be
given by (\ref{eq:FN def section 2}).For each $1\leq j\leq k$ and
$1\leq i\leq m_{j}$ set 
\[
V_{L}^{i,j}\coloneqq\left\{ (t,x)\left|t\in[T_{j-1},T_{j}),x_{(i-1)^{\ast}(j)}(t)\leq x<x_{i^{\ast}(j)}(t)\right.\right\} 
\]
 and 
\[
V_{R}^{i,j}\coloneqq\left\{ (t,x)\left|t\in[T_{j-1},T_{j}),x_{i^{\ast}(j)}(t)<x\leq x_{(i+1)^{\ast}(j)}(t)\right.\right\} .
\]
(with the convention $x_{0^{\ast}}=-\infty$ and $x_{(m_{j}+1)^{\ast}}=+\infty$.
Then, for each $1\leq j\leq k$ and $1\leq i\leq m_{j}$ , $F_{N}$
is a classical solution on both $V_{L}^{i,j}$ and $V_{R}^{i,j}$.
\end{lem}
\begin{proof}
For any $(t,x)\in V_{L}^{i,j}$ we have 

\[
F_{N}(t,x)=-\frac{1}{2}+\frac{1}{N}\stackrel[r=1]{(i-1)^{\ast}(j)}{\sum}m_{r}(t)=\theta_{(i-1)^{\ast}(j)}(t)
\]
and for any $(t,x)\in V_{R}^{i,j}$ we have 

\[
F_{N}(t,x)=-\frac{1}{2}+\frac{1}{N}\stackrel[r=1]{i^{\ast}(j)}{\sum}m_{r}(t)=\theta_{i^{\ast}(j)}(t).
\]
 In particular, note that $A_{N}(t,F_{N}(t,x))$ is constant in $x$
on both regions, so that its $x$-derivative vanishes. In addition
for each $x_{(i-1)^{\ast}(j)}(t)\leq x<x_{i^{\ast}(j)}(t)$ and $t\in[T_{j-1},T_{j})$
a routine calculation shows that 
\begin{align*}
    \mathbf{S}[F_{N}](t,x)&=\frac{1}{N^{2}}\underset{l,r}{\sum}H(x-x_{l}(t))m_{l}(t)m_{r}(t)S(x_{l}(t)-x_{r}(t))
    \\
    &=\frac{1}{N^{2}}\stackrel[r=1]{N}{\sum}\stackrel[l=1]{(i-1)^{\ast}(j)}{\sum}m_{l}(t)m_{r}(t)S(x_{l}(t)-x_{r}(t)),
\end{align*}
while 
\begin{align*}
\frac{d}{dt}F_{N}(t,x)&=\frac{1}{N^{2}}\stackrel[l=1]{(i-1)^{\ast}(j)}{\sum}\stackrel[r=1]{N}{\sum}m_{l}(t)m_{r}(t)S(x_{l}(t)-x_{r}(t))
\\
&=\frac{1}{N^{2}}\stackrel[r=1]{N}{\sum}\stackrel[l=1]{(i-1)^{\ast}(j)}{\sum}m_{l}(t)m_{r}(t)S(x_{l}(t)-x_{r}(t))=\mathbf{S}[F_{N}].
\end{align*}
This shows that $F_{N}$ is a classical solution on $V_{L}^{i,j}$.
The same calculation shows that $F_{N}$ is a classical solution on
$V_{R}^{i,j}$. 
\end{proof}
We move to verify conditions (ii) and (iii). 
\begin{lem}
\label{application of rankine hugoniot } The function $F_{N}$ is
an entropy solution to Equation (\ref{discretized Burgers}) in the
sense of Definition (\ref{definition of kruzkov solution-1})
\end{lem}
\begin{proof}
Condition (1) is from Lemma \ref{Long time constant } and that $x\mapsto F_{N}(t,x)$
is non-decreasing is immediate from the fact that the weights are
positive, which gives condition (2). Validating the requested weak
inequality is slightly longer and rests upon the Rankine-Hugoniot
and Oleinik conditions (which as already remarked, are recapped in
Appendix \ref{Appendix A}). With the same notation of Lemma \ref{FN is classical solution on a partition },
keep $1\leq j\leq k$ and $1\leq i\leq m_{j}$ fixed and consider
the curve 

\[
\Gamma_{i,j}\coloneqq\left\{ (t,x_{i^{\ast}(j)}(t))\left|t\in[T_{j-1},T_{j})\right.\right\} .
\]

\textbf{The Rankine-Hugoniot Condition}. We wish to show that 
\[
\frac{A(t,F_{N,L}^{i,j}(t))-A(t,F_{N,R}^{i,j}(t))}{F_{N,L}^{i,j}(t)-F_{N,R}^{i,j}(t)}=\dot{x}_{i^{\ast}(j)}(t).
\]
To make the equations a bit lighter let us abbreviate $i^{\ast}(j)=\overline{i},(i-1)^{\ast}(j)=\underline{i}$.
It is clear that 
\[
F_{N,L}^{i,j}(t)=-\frac{1}{2}+\frac{1}{N}\stackrel[r=1]{\underline{i}}{\sum}m_{r}(t)=\theta_{\underline{i}}(t)
\]
and 
\[
F_{N,R}^{i,j}(t)=-\frac{1}{2}+\frac{1}{N}\stackrel[r=1]{\overline{i}}{\sum}m_{r}(t)=\theta_{\overline{i}}(t).
\]
Since $\frac{1}{N}\stackrel[r=1]{N}{\sum}m_{r}(t)=1$ we may rewrite 

\[
F_{N,L}^{i,j}(t)=\frac{1}{2}-\frac{1}{N}\underset{r>\underline{i}}{\sum}m_{r}(t).
\]
Therefore, using that $\frac{1}{N}\stackrel[j=1]{N}{\sum}m_{j}(t)=1$
and the fact that the initial order is preserved, the following identites
hold 
\begin{align*}
\frac{A_{N}(t,F_{N,L}^{i,j}(t))-A_{N}(t,F_{N,R}^{i,j}(t))}{F_{N,L}^{i,j}(t)-F_{N,R}^{i,j}(t)}&=\frac{\theta_{\overline{i}}^{2}(t)-\theta_{\underline{i}}^{2}(t)}{\theta_{\underline{i}}(t)-\theta_{\overline{i}}(t)}=-(\theta_{\overline{i}}(t)+\theta_{\underline{i}}(t))\\
&=1-\frac{1}{N}\overset{\overline{i}}{\underset{r=1}{\sum}}m_{r}(t)-\frac{1}{N}\overset{\underline{i}}{\underset{r=1}{\sum}}m_{r}(t)\\
&=1-\left(1-\frac{1}{N}\underset{r>\overline{i}}{\sum}m_{r}(t)\right)-\frac{1}{N}\overset{\underline{i}}{\underset{r=1}{\sum}}m_{r}(t)=\frac{1}{N}\underset{r>\overline{i}}{\sum}m_{r}(t)-\frac{1}{N}\overset{\underline{i}}{\underset{r=1}{\sum}}m_{r}(t)\\
&=\frac{1}{N}\underset{r>\overline{i}}{\sum}\mathrm{sgn}(x_{r}(t)-x_{\overline{i}}(t))m_{r}(t)+\frac{1}{N}\overset{\underline{i}}{\underset{r=1}{\sum}}\mathrm{sgn}(x_{r}(t)-x_{\overline{i}}(t))m_{r}(t)\\
&\quad+\frac{1}{N}\overset{\overline{i}}{\underset{r=\underline{i}+1}{\sum}}\mathrm{sgn}(x_{r}(t)-x_{\overline{i}}(t))m_{r}(t)\\
&=\frac{1}{N}\stackrel[r=1]{N}{\sum}m_{r}(t)\mathrm{sgn}(x_{r}(t)-x_{\overline{i}}(t))=\dot{x}_{\overline{i}}(t).    
\end{align*}
\textbf{The Oleinik Condition}. Let $\theta\in(\theta_{\underline{i}}(t),\theta_{\overline{i}}(t))$.
We wish to show 

\[
\frac{A_{N}(t,\theta)-A_{N}(t,\theta_{\underline{i}}(t)))}{\theta-\theta_{\underline{i}}(t)(t)}\geq\dot{x}_{\overline{i}}(t).
\]

Note that if we pick $\underline{i}+1\leq m\leq\overline{i}$ such
that $\theta\in[\theta_{m-1}(t),\theta_{m}(t)]$ then since $A_{N}$
is piecewise linear we clearly have 

\[
\frac{A_{N}(t,\theta)-A_{N}(t,\theta_{\underline{i}}(t)))}{\theta-\theta_{\underline{i}}(t)}\geq\min\left\{ \frac{A_{N}(t,\theta_{m}(t))-A_{N}(t,\theta_{\underline{i}}(t)))}{\theta_{m}(t)-\theta_{\underline{i}}(t)},\frac{A_{N}(t,\theta_{m-1}(t))-A_{N}(t,\theta_{\underline{i}}(t)))}{\theta_{m-1}(t)-\theta_{\underline{i}}(t)}\right\} .
\]
Therefore it suffices to check the inequality for $\theta=\theta_{k}(t)$
where $\underline{i}\leq k\leq\overline{i}$. We have

\[\frac{A_{N}(t,\theta_{k}(t))-A_{N}(t,\theta_{\underline{i}}(t)))}{\theta_{k}(t)-\theta_{\underline{i}}(t)}
=\frac{A(\theta_{k}(t))-A(\theta_{\underline{i}}(t))}{\theta_{k}(t)-\theta_{\underline{i}}(t)}=-(\theta_{k}(t)+\theta_{\underline{i}}(t))
\geq-(\theta_{\overline{i}}(t)+\theta_{\underline{i}}(t))=\dot{x}_{\overline{i}}(t)
\]
where the inequality is because $\theta_{\overline{i}}(t)\geq\theta_{k}(t)$,
and the last identity is a byproduct of the previous step. 
\end{proof}

\section{\label{sec:The-Mean-Field} The Mean Field Limit }

In this section we prove well-posedness, stability estimates and mean
field limit for the equation 

\begin{equation}
\partial_{t}F+\partial_{x}(A(t,F))=\mathbf{S}[F](t,x)\label{burgers section 4}
\end{equation}
where 

\[
\mathbf{S}[F](t,x)\coloneqq F(t,x)(\phi\star F)(t,x)-\int^x_{-\infty}F(t,z)(\partial_{z}\phi\star F)(t,z)dz,\ \phi\coloneqq\partial_{x}S.
\]
This equation has a slightly more general form than the Equation \ref{eq:-3}, since here the flux is time dependent. As already explained, the strategy of proof is a modification of the
argument in \cite{11}, and can be divided into the following steps

1. Extraction of a converging subsequence from $F_{N}$ with a limit
$F$. 

2. Showing that the limit $F$ obtained in 1. is an entropy solution. 

3. Provided steps 1+2 are successfully established, it remains to
prove stability estimates (from which the remaining parts- that is
uniqueness and mean field limit- would follow). This is one step where
our argument differs from the one in \cite{11}, since the source
term in question carries a different form. Some technical modifications
appear in step 2 as well. 

Before detailing the proof of the plan proposed above, we recall the
following chain rule for $BV$ functions, which will be used in the
course of the proof 
\begin{lem}
\textup{(\cite{5}, Lemma A.21)} \label{Chain rule } Suppose $W\in BV_{\mathrm{loc}}(\mathbb{R})$
and $f$ is Lipschitz. Then $f\circ W\in BV_{\mathrm{loc}}(\mathbb{R})$
and 

\[
\left|\frac{d}{dx}f\circ W\right|\leq\left|f\right|_{\mathrm{Lip}}\left|\frac{d}{dx}W\right|
\]
in the sense of measures. 
\end{lem}
In order to treat the source term we will be forced to verify the
entropy inequality only on dense subset of the real line, which is
sufficient, as observed in the following simple 
\begin{lem}
\label{entropy condition on a dense set } The entropy inequality
holds (\ref{entropy inequality}) iff it holds for some dense set
$D\subset\mathbb{R}$. 
\end{lem}
\begin{proof}
Let $D\subset\mathbb{R}$ be dense and suppose (\ref{entropy inequality})
holds all $\beta\in D$, and let $\alpha\in\mathbb{R}$. Then, taking
a sequence $\left(\beta_{k}\right)\subset D$ such that $\beta_{k}\rightarrow\alpha$,
we have that for all $k$ 

\begin{align*}
&\int_0^T\int_\mathbb{R}\mathbf{1}_{F>\beta_{k}}\left((F-\beta_{k})\chi_{t}+\left(A(F)-A(\beta_{k})\right)\partial_{x}\chi+\chi\mathbf{S}(F)(t,x)\right)dxdt \\
&\quad-\int_0^T\int_\mathbb{R}\mathbf{1}_{F<\beta_{k}}\left((F-\beta_{k})\chi_{t}+\left(A(F)-A(\beta_{k})\right)\partial_{x}\chi+\chi\mathbf{S}(F)(t,x)\right)dxdt .
\end{align*}
As $k\rightarrow\infty$ the first integral tends to 
\[
\int_0^T\int_\mathbb{R}\mathbf{1}_{F\geq\alpha}\left((F-\alpha)\chi_{t}+\left(A(F)-A(\alpha)\right)\partial_{x}\chi+\chi\mathbf{S}(F)(t,x)\right)dxdt,
\]
whereas the second integral tends to 

\[
-\int_0^T\int_\mathbb{R}\mathbf{1}_{F\leq\alpha}\left((F-\alpha)\chi_{t}+\left(A(F)-A(\alpha)\right)\partial_{x}\chi+\chi\mathbf{S}(F)(t,x)\right)dxdt,
\]
so that 
\begin{align*}
&\int_0^T\int_\mathbb{R}\left(\left|F-\alpha\right|\chi_{t}+\mathrm{sgn}(F-\alpha)\left(A(F)-A(\alpha)\right)\partial_{x}\chi+\mathrm{sgn}(F-\alpha)\chi\mathbf{S}(F)(t,x)\right)dxdt \\
&\quad=\int_0^T\int_\mathbb{R}\mathbf{1}_{F\geq\alpha}\left((F-\alpha)\chi_{t}+\left(A(F)-A(\alpha)\right)\partial_{x}\chi+\chi\mathbf{S}(F)(t,x)\right)dxdt \\
&\qquad-\int_0^T\int_\mathbb{R}\mathbf{1}_{F\leq\alpha}\left((F-\beta_{k})\chi_{t}+\left(A(F)-A(\beta_{k})\right)\partial_{x}\chi+\chi\mathbf{S}(F)(t,x)\right)dxdt\geq0.
\end{align*}
\end{proof}
We are now in a good position to prove the main result, which we now
state. Remark that we insist on including time dependency in the flux
in order to enable applying the belowstated stability estimate for
the fluxes $A_{N}(t,x)$ and $A(x)$ as defined in Section \ref{sec:The-Discretized-Version}.
\begin{thm}
\label{Main Theorem } Let $S\in C_{0}^{\infty}(\mathbb{R})$ be such
$\mathrm{supp}(S)\subset[-r,r]$ for some $r>0$. Let $A(t,x),\widetilde{A}(t,x)\in C([0,T];\mathrm{Lip}([-\frac{1}{2},\frac{1}{2}]))$.
Suppose $F^{\mathrm{0}}\in BV(\mathbb{R})$ is non-decreasing and
there is some $R>0$ such that 
\begin{equation}
\forall x\geq R:F^{\mathrm{0}}(\pm x)=\pm\frac{1}{2}.\label{eq:assumption on initial data}
\end{equation}
1. (Well-posedness) There exist a unique entropy solution (in the
sense of Definition (\ref{definition of kruzkov solution-1})) to
the problem (\ref{burgers section 4}). 

2. (Stability) If $F,\widetilde{F}$ are two entropy solutions with
initial datas $F^{\mathrm{0}},\widetilde{F^{\mathrm{0}}}$ respectively
satisfying (\ref{eq:assumption on initial data}) then there is some
$C=C(r,R,\left\Vert \partial_{z}\phi\right\Vert _{\infty},\left\Vert \phi\right\Vert _{\infty})>0$
such that 

\[
\left\Vert F(t,\cdot)-\widetilde{F}(t,\cdot)\right\Vert _{1}\leq e^{Ct}\left(\left\Vert F^{\mathrm{0}}-\widetilde{F^{\mathrm{0}}}\right\Vert _{1}+t\underset{t\in[0,T]}{\sup}\left|A(t,\cdot)-\widetilde{A}(t,\cdot)\right|_{\mathrm{Lip}}\right).
\]
3. (Mean field limit) It holds that $F_{N}-F\underset{N\rightarrow\infty}{\rightarrow}0$
in $C([0,T];L^{1}(\mathbb{R}))$ provided $F_{N}^{\mathrm{0}}-F^{\mathrm{0}}\rightarrow0$
in $L^{1}(\mathbb{R})$. 
\end{thm}
\textit{Proof}. \textbf{Step 1. Extracting a converging subsequence
from $F_{N}$. }Keep $t\in[0,T]$ fixed . We know that $F_{N}(t,\cdot)$
is a sequence of non-decreasing functions with uniformly bounded (with
respect to $N$) total variation on each compact subset. Therefore,
by Helly's selection theorem there is a subsequence, still labled
$F_{N}(t,\cdot)$, and some $F(t,\cdot)\in L_{\mathrm{loc}}^{1}(\mathbb{R})$
such that $F_{N}(t,\cdot)\underset{N\rightarrow\infty}{\rightarrow}F(t,\cdot)$
pointwise a.e. By diagonalization we may find a subsequence, still
labled $F_{N}(t)$, such that $F_{N}(t,\cdot)$ converge to $F(t,\cdot)$
for all $t\in\mathbb{Q}$. Recall that by Lemma \ref{Long time constant },
as long as $x>R$, $F_{N}(t,\pm x)=\pm\frac{1}{2}$ for all $t$.
As a result, the same conclusion is true for the limit $F(t,\cdot)$,
so that we have $F_{N}(t,\cdot)-F(t,\cdot)\underset{N\rightarrow\infty}{\rightarrow}0$
in $L^{1}(\mathbb{R})$ for all $t\in\mathbb{Q}$. Next, we upgrade
the convergence to irrational times as well. 
\begin{align*}
  \int_\mathbb{R}\left|F_{N}(t,x)-F_{N}(s,x)\right|dx &=\int_{\left|x\right|\leq R}\left|\frac{1}{N}\stackrel[i=1]{N}{\sum}m_{i}(t)\mathbf{1}_{\left\{ x\left|x\geq x_{i}(t)\right.\right\} }-\frac{1}{N}\stackrel[i=1]{N}{\sum}m_{i}(s)\mathbf{1}_{\left\{ x\left|x\geq x_{i}(s)\right.\right\} }\right|dx\\
   &\leq\int_{\left|x\right|\leq R}\left|\frac{1}{N}\stackrel[i=1]{N}{\sum}m_{i}(t)\mathbf{1}_{\left\{ x\left|x_{i}(t)\leq x\leq x_{i}(s)\right.\right\} }\right|dx \\
   &\quad+\int_{\left|x\right|\leq R}\left|\frac{1}{N}\stackrel[i=1]{N}{\sum}(m_{i}(t)-m_{i}(s))\mathbf{1}_{\left\{ x\left|x\geq x_{i}(s)\right.\right\} }\right|dx.
\end{align*}
The first integral is 
\[
\leq\overline{M}\underset{1\leq i\leq N}{\max}\left|x_{i}(t)-x_{i}(s)\right|\leq\overline{M}^{2}\left|t-s\right|,
\]
whereas the second integral is 

\[
\leq2R\underset{1\leq i\leq N}{\max}\left|m_{i}(t)-m_{i}(s)\right|\leq2RC(1+\overline{M})\left|t-s\right|,
\]
so that 
\begin{equation}
\left\Vert F_{N}(t,\cdot)-F_{N}(s,\cdot)\right\Vert _{1}\leq c\left|t-s\right|,\label{eq:-4}
\end{equation}
for some suitable constant $c=c(C,\overline{M},R)$. The estimate
(\ref{eq:-4}) implies that $F_{N}(t,\cdot)$ is a Cauchy sequence
in $L^{1}((-R,R))$ for any $t\in[0,T]$ . Indeed, keep $t$ fixed
and let $\left\{ t_{k}\right\} _{k=1}^{\infty}\subset\mathbb{Q}$
such that $t_{k}\underset{k\rightarrow\infty}{\rightarrow}t$ . Let
$\varepsilon>0$. We estimate 
\begin{align*}
\left\Vert F_{N}(t,\cdot)-F_{N+p}(t,\cdot)\right\Vert _{1}
&\leq\left\Vert F_{N}(t,\cdot)-F_{N}(t_{k},\cdot)\right\Vert _{1}+\left\Vert F_{N}(t_{k},\cdot)-F(t_{k},\cdot)\right\Vert _{1}\\
&\quad+\left\Vert F_{N+p}(t,\cdot)-F_{N+p}(t_{k},\cdot)\right\Vert _{1}+\left\Vert F_{N+p}(t_{k},\cdot)-F(t_{k},\cdot)\right\Vert _{1} \\
&\leq2c\left|t-t_{k}\right|+\left\Vert F_{N}(t_{k},\cdot)-F(t_{k},\cdot)\right\Vert _{1}+\left\Vert F_{N+p}(t_{k},\cdot)-F(t_{k},\cdot)\right\Vert _{1}.    
\end{align*}
Pick $k$ large enough so that $\left|t-t_{k}\right|<\varepsilon$
and pick $N_{0}$ large enough so that for any $N\geq N_{0}$ one
has $\left\Vert F_{N}(t_{k},\cdot)-F(t_{k},\cdot)\right\Vert _{1}<\varepsilon.$
Then for these choices we get 

\[
\left\Vert F_{N}(t,\cdot)-F_{N+p}(t,\cdot)\right\Vert \lesssim\varepsilon
\]
and therefore there is some $\widetilde{F}(t,\cdot)\in L^{1}((-R,R))$
such that $F_{N}(t,\cdot)\underset{N\rightarrow\infty}{\rightarrow}\widetilde{F}(t,\cdot)$
in $L^{1}((-R,R))$. In particular, for each $t$, the sequence $F_{N}(t,\cdot)$
is confined in a compact set of $L^{1}((-R,R))$. Therefore the estimate
(\ref{eq:-4}) makes the theorem of Arzel\'a-Ascoli available (Theorem
1.1 in \cite{17} for example), thereby ensuring the existence of
a subsequence, still labled $F_{N}$, and an element $\widetilde{\widetilde{F}}\in C([0,T];L_{\mathrm{loc}}^{1}(\mathbb{R}))$
such that $F_{N}-\widetilde{\widetilde{F}}\underset{N\rightarrow\infty}{\rightarrow}0$
in $C([0,T];L^{1}(\mathbb{R}))$. It is clear that $\widetilde{\widetilde{F}}\in BV([0,T]\times\mathbb{R})$
and $\forall x\geq R:\widetilde{\widetilde{F}}(t,\pm x)=\pm\frac{1}{2}$.
To minimize cumbersome notation we shall hereafter denote by $F$
the limit function $\widetilde{\widetilde{F}}$ that we constructed. 

\textbf{Step 2. The Limit function $F$ is an entropy solution. }We
wish to show that the limit function $F$ obtained in step 1. is an
entropy solution to the original equation, which will establish the
existence claim. By Lemma \ref{application of rankine hugoniot }
we know that 
\[
\int_0^T\int_\mathbb{R}\left|F_{N}-\alpha\right|\chi_{t}+\mathrm{sgn}(F_{N}-\alpha)\left(A_{N}(t,F_{N})-A_{N}(t,\alpha)\right)\partial_{x}\chi+\mathrm{sgn}(F_{N}-\alpha)\chi\mathbf{S}[F_{N}](t,x)dxdt\geq0,
\]
and we wish now to pass to the limit as $N\rightarrow\infty$. That 

\[
\int_0^T\int_\mathbb{R}\left|F_{N}-\alpha\right|\chi_{t}\rightarrow\int_0^T\int_\mathbb{R}\left|F-\alpha\right|\chi_{t}
\]
is due to dominated convergence theorem. In addition, Lemma \ref{AN approximates A}
and the identity $\mathrm{sgn}(a-b)(a^{2}-b^{2})=\left|a-b\right|(a+b)$
entails 
\begin{align*}
  &\left|\int_0^T\int_\mathbb{R}\mathrm{sgn}(F_{N}-\alpha)\left(A_{N}(t,F_{N})-A_{N}(t,\alpha)\right)\partial_{x}\chi-\int_0^T\int_\mathbb{R}\mathrm{sgn}(F-\alpha)\left(A(F)-A(\alpha)\right)\partial_{x}\chi\right|\\ 
  &\quad\leq\left|\int_0^T\int_\mathbb{R}\mathrm{sgn}(F_{N}-\alpha)\left(A_{N}(t,F_{N})-A_{N}(t,\alpha)\right)\partial_{x}\chi-\int_0^T\int_\mathbb{R}\mathrm{sgn}(F_{N}-\alpha)\left(A(F_{N})-A(\alpha)\right)\partial_{x}\chi\right|\\
  &\qquad+\left|\int_0^T\int_\mathbb{R}\mathrm{sgn}(F_{N}-\alpha)\left(A(F_{N})-A(\alpha)\right)\partial_{x}\chi-\int_0^T\int_\mathbb{R}\mathrm{sgn}(F-\alpha)\left(A(F)-A(\alpha)\right)\partial_{x}\chi\right|\\
  &\quad\leq\frac{4\overline{M}\mathrm{Lip}(\left.A\right|_{[-\frac{1}{2},\frac{1}{2}]})}{N}\int_0^T\int_\mathbb{R}\left|\partial_{x}\chi\right|dxdt \\
  &\qquad+\left|\int_0^T\int_\mathbb{R}\left|F_{N}-\alpha\right|(F_{N}+\alpha)\partial_{x}\chi-\int_0^T\int_\mathbb{R}\left|F-\alpha\right|(F+\alpha)\partial_{x}\chi\right|\underset{N\rightarrow\infty}{\rightarrow}0.
\end{align*}
The last term is slightly more subtle. Observe that 
\[
\int_{-2R}^{2R}\phi(z-\zeta)F_{N}(t,\zeta)d\zeta\underset{N\rightarrow\infty}{\rightarrow}\int_{-2R}^{2R}\phi(z-\zeta)F(t,\zeta)d\zeta,
\]
so that 
\[
F_{N}(t,x)\phi\star F_{N}(t,x)\underset{N\rightarrow\infty}{\rightarrow}F(t,x)\phi\star F(t,x)
\]
pointwise a.e. Recall that $\phi'\star F_{N}$ is compactly supported
and therefore dominated convergence is applicable for the second integral
as well which implies 
\[
\int^x_{-\infty}F_{N}(t,z)\phi'\star F_{N}(t,z)dz\underset{N\rightarrow\infty}{\rightarrow}\int^x_{-\infty}F(t,z)\phi'\star F(t,z)dz.
\]
By Lemma (\ref{entropy condition on a dense set }) it suffices to
verify the entropy condition on dense set. Note that since $F$ is
locally summable it holds that $\lambda(\left\{ (t,x)\left|F(t,x)=\alpha\right.\right\} )=0$
for a.e. $\alpha$. In particular there is dense set $D\subset\mathbb{R}$
such that $\lambda(\left\{ F(t,x)=\alpha\right\} )=0$ for all $\alpha\in D$.
Therefore, since $F_{N}\rightarrow F$ and $\mathbf{S}[F_{N}](t,x)\rightarrow\mathbf{S}[F](t,x)$
pointwise a.e., we conclude that for all $\alpha\in D$ 
\begin{align*}
  \int_0^T\int_\mathbb{R}\mathrm{sgn}(F_{N}-\alpha)\chi\mathbf{S}[F_{N}](t,x)dxdt
  &=\int_0^T\int_\mathbb{R}\mathbf{1}_{\left\{ F(t,x)\neq\alpha\right\} }\mathrm{sgn}(F_{N}-\alpha)\chi\mathbf{S}[F_{N}](t,x)dxdt\\
  &\underset{N\rightarrow\infty}{\rightarrow}\int_0^T\int_\mathbb{R}\mathrm{sgn}(F-\alpha)\chi\mathbf{S}[F](t,x)dxdt  
\end{align*}
as desired. 

\textbf{Step 3.} \textbf{Uniqueness and Stability.} Let $F,\widetilde{F}$
be two entropy solutions. Keep $(s,y)\in[0,T]\times\mathbb{R}$ fixed.
Then by definition we have 
\begin{align*}
  0 &\le \iint\left|F(t,x)-\widetilde{F}(s,y)\right|\partial_{t}\chi(t,x,s,y)dxdt\\
  &\quad+\iint\mathrm{sgn}(F(t,x)-\widetilde{F}(s,y))\left(A(t,F(t,x))-A(t,\widetilde{F}(s,y))\right)\partial_{x}\chi(t,x,s,y)dxdt\\
  &\quad+\iint\mathrm{sgn}(F(t,x)-\widetilde{F}(s,y))\chi(t,x,s,y)\mathbf{S}[F](t,x)dxdt.  
\end{align*}
Exchanging the roles of $F$ and $\widetilde{F}$ we also have the
inequality 
\begin{align*}
0 &\le \iint\left|\widetilde{F}(s,y)-F(t,x)\right|\partial_{s}\chi(t,x,s,y)dyds\\
&\quad+\iint\mathrm{sgn}(\widetilde{F}(s,y)-F(t,x))\left(\widetilde{A}(s,\widetilde{F}(s,y))-\widetilde{A}(s,F(t,x))\right)\partial_{y}\chi(t,x,s,y)dyds\\
&\quad+\iint\mathrm{sgn}(\widetilde{F}(s,y)-F(t,x))\chi(t,x,s,y)\mathbf{S}[\widetilde F](s,y)dyds.
\end{align*}
Integrating both of the above inequalities over the free variables
and summing up gives
\begin{align*}
0&\le \iiiint\left|F(t,x)-\widetilde{F}(s,y)\right|\left(\partial_{t}\chi+\partial_{s}\chi\right) \\
&\quad+\iiiint\mathrm{sgn}(F(t,x)-\widetilde{F}(s,y))\left(A(t,F(t,x))-A(t,\widetilde{F}(s,y))\right)\partial_{x}\chi\\
&\quad+\iiiint\mathrm{sgn}(\widetilde{F}(s,y)-F(t,x))\left(\widetilde{A}(s,\widetilde{F}(s,y))-\widetilde{A}(s,F(t,x))\right)\partial_{y}\chi\\
&\quad+\iiiint\left(\chi\mathrm{sgn}(F(t,x)-\widetilde{F}(s,y))\mathbf{S}[F](t,x)+\chi\mathrm{sgn}(\widetilde{F}(s,y)-F(t,x))\mathbf{S}[\widetilde{F}](s,y)\right)dxdtdyds.
\end{align*}
\begin{equation}
\label{eq:}
\end{equation}
Now consider the variable change 

\[
\overline{y}=\frac{x-y}{2},\quad \overline{s}=\frac{t-s}{2},\quad \overline{x}=\frac{x+y}{2},\quad \overline{t}=\frac{t+s}{2},
\]
and take $\chi$ to be of the form 
\begin{align*}
\chi(t,x,s,y) &=b_{\varepsilon}\left(\frac{x-y}{2}\right)b_{\varepsilon}\left(\frac{t-s}{2}\right)g\left(\frac{x+y}{2}\right)h_{\delta}\left(\frac{t+s}{2}\right) \\
&=b_{\varepsilon}\left(\overline{y}\right)b_{\varepsilon}\left(\overline{s}\right)g\left(\overline{x}\right)h_{\delta}\left(\overline{t}\right).
\end{align*}
where $b_{\varepsilon},g,h_{\delta}$ are chosen as follows. Keep
$\sigma<\tau\in(0,T)$ fixed. For each $\varepsilon>0,\delta>0$ such
that $0<\varepsilon+\delta<\min(\sigma,T-\tau)$ we define the functions
as follows. 

1. $b_{\varepsilon}$ is an approximation of the identity, i.e. $b_{\varepsilon}=\frac{1}{\varepsilon}\zeta(\frac{x}{\varepsilon})$
for some radial $0\leq\zeta\in C_{0}^{\infty}(\mathbb{R})$ with $\mathrm{supp}(\zeta)\Subset B_{1}(0)$. 

2. $g(x)=1$ for all $x\in[-R,R]$ and $g\in C_{0}^{\infty}(\mathbb{R})$.

3. Let 
\[
h_{\delta}(t)\coloneqq\left\{ \begin{array}{c}
1,\ \ t\in[\sigma,\tau]\\
-\frac{1}{\delta}t+\frac{\tau+\delta}{\delta},\ \ t\in[\tau,\tau+\delta]\\
\frac{1}{\delta}t-\frac{\sigma-\delta}{\delta},\ \ t\in[\sigma-\delta,\sigma]\\
0,\ \ t\notin[\sigma-\delta,\tau+\delta]
\end{array}\right..
\]
With this choice we easily observe that $\chi(t,x,s,y)$ is compactly
supported in the variables $(t,x)$ for any fixed $(s,y)$, and viceversa. 
\begin{claim}
For each fixed $(s,y)\in(0,T)\times\mathbb{R}$ the function $(t,x)\mapsto\chi(t,x,s,y)$
is compactly supported in $(0,T)\times\mathbb{R}$. Vice versa, for
each fixed $(t,x)\in(0,T)\times\mathbb{R}$ the function $(s,y)\mapsto\chi(t,x,s,y)$
is compactly supported in $(0,T)\times\mathbb{R}$. 
\end{claim}
\begin{proof}
That $x\mapsto\chi(t,x,s,y)$ vanishes outside some finite interval
is clear. In addition if $s>2\varepsilon$ then $b_{\varepsilon}(\frac{t-s}{2})=0$
for all $t\in(0,s-2\varepsilon)$, while if $s\leq2\varepsilon$ then
$h_{\delta}(\frac{t+s}{2})=0$ then since $\sigma>\varepsilon+\delta$
it follows that $2\sigma-2\delta-s>0$, so that $h_{\delta}(\frac{t+s}{2})=0$
for all $t\in(0,2\sigma-2\delta-s)$. Hence $\chi(t,x,s,y)=0$ for
all $t\in(0,\mathbf{1}_{s<2\varepsilon}(s-2\varepsilon)+\mathbf{1}_{s\geq2\varepsilon}2\sigma-2\delta-s)$
(note that the right end of the interval is strictly positive). If
$s<T-2\varepsilon$ then $b_{\varepsilon}(\frac{t-s}{2})=0$ for all $t\in(s+2\varepsilon,T)$
while if $s\geq T-2\varepsilon$ then $2\tau+2\delta-s\leq2\tau+2\delta+2\varepsilon-T<2T-T=T$
because $\tau<T-(\delta+\varepsilon)$. Therefore $\chi(t,x,s,y)=0$
for all $t\in(\mathbf{1}_{s<T-2\varepsilon}(s+2\varepsilon)+\mathbf{1}_{s\geq T-2\varepsilon}(2\tau+2\delta-s),T)$.
To conclude $t\mapsto\chi(t,x,s,y)$ vanishes outside some interval
compactly supported in $(0,T)$. Since $b_{\varepsilon}$ is radial,
by symmetry the same argument shows that for any fixed $(t,x)$ the
function $(s,y)\mapsto\chi(t,x,s,y)$ is compactly supported. 
\end{proof}
In addition, it is readily checked that 

\[
\partial_{t}+\partial_{s}=\partial_{\overline{t}},\quad \partial_{x}+\partial_{y}=\partial_{\overline{x}},\quad \partial_{x}-\partial_{y}=\partial_{\overline{y}}.
\]
Under this variable change, the second and third term in \ref{eq:}
can be grouped together as follows (to make the equations lighter
the arguments are of $F$ and $\widetilde{F}$ are always $(\overline{t}+\overline{s},\overline{x}+\overline{y})$
and $(\overline{t}-\overline{s},\overline{x}-\overline{y})$ respectively,
and are implicit )
\begin{align*}
  &\iiiint\mathrm{sgn}(F(t,x)-\widetilde{F}(s,y))\left(A(t,F(t,x))-A(t,\widetilde{F}(s,y))\right)\partial_{x}\chi \\
  &\qquad+\iiiint\mathrm{sgn}(\widetilde{F}(s,y)-F(t,x))\left(\widetilde{A}(s,\widetilde{F}(s,y))-\widetilde{A}(s,F(t,x))\right)\partial_{y}\chi\\
  &\quad=\frac{1}{2}\iiiint\mathrm{sgn}(F-\widetilde{F})\left(A(\overline{t}+\overline{s},F)-A(\overline{t}+\overline{s},\widetilde{F})\right)(\partial_{\overline{x}}+\partial_{\overline{y}})\chi\\
  &\qquad+\frac{1}{2}\iiiint\mathrm{sgn}(\widetilde{F}-F)\left(\widetilde{A}(\overline{t}-\overline{s},\widetilde{F})-\widetilde{A}(\overline{t}-\overline{s},F)\right)(\partial_{\overline{x}}-\partial_{\overline{y}})\chi\\
  &\quad=\frac{1}{2}\iiiint\mathrm{sgn}(F-\widetilde{F})\left(A(\overline{t}+\overline{s},F)+\widetilde{A}(\overline{t}-\overline{s},F)-A(\overline{t}+\overline{s},\widetilde{F})-\widetilde{A}(\overline{t}-\overline{s},\widetilde{F})\right)\partial_{\overline{x}}\chi\\
  &\qquad+\frac{1}{2}\iiiint\mathrm{sgn}(F-\widetilde{F})\left(A(\overline{t}+\overline{s},F)-\widetilde{A}(\overline{t}-\overline{s},F)-A(\overline{t}+\overline{s},\widetilde{F}(s,y))+\widetilde{A}(\overline{t}-\overline{s},\widetilde{F}(s,y))\right)\partial_{\overline{y}}\chi\\
  &\quad=\iiiint\mathrm{sgn}(F-\widetilde{F})\left(A_{+}(\overline{t}+\overline{s},\overline{t}-\overline{s},F)-A_{+}(\overline{t}+\overline{s},\overline{t}-\overline{s},\widetilde{F})\right)\partial_{\overline{x}}\chi\\
  &\qquad+\iiiint\mathrm{sgn}(F-\widetilde{F})\left(A_{-}(\overline{t}+\overline{s},\overline{t}-\overline{s},F)-A_{-}(\overline{t}+\overline{s},\overline{t}-\overline{s},\widetilde{F})\right)\partial_{\overline{y}}\chi,
\end{align*}
where we have set

\[
A_{\pm}(t,s,x)\coloneqq\frac{A(t,x)\pm\widetilde{A}(s,x)}{2}.
\]
The Inequality \ref{eq:} is then transformed to 
\begin{equation}
\begin{aligned}
0 &\le I_{\varepsilon,\delta}+II_{\varepsilon,\delta}+III_{\varepsilon,\delta}+IV_{\varepsilon,\delta}\\
&\coloneqq\iiiint\left|F(\overline{t}+\overline{s},\overline{x}+\overline{y})-\widetilde{F}(\overline{t}-\overline{s},\overline{x}-\overline{y})\right|\partial_{\overline{t}}\chi(\overline{t},\overline{x},\overline{s},\overline{y})d\overline{t}d\overline{x}d\overline{s}d\overline{y} \\
&\quad+\iiiint\mathrm{sgn}(F(\overline{t}+\overline{s},\overline{x}+\overline{y})-\widetilde{F}(\overline{t}-\overline{s},\overline{x}-\overline{y}))\left(A_{+}(\overline{t}+\overline{s},\overline{t}-\overline{s},F)-A_{+}(\overline{t}+\overline{s},\overline{t}-\overline{s},\widetilde{F})\right)\partial_{\overline{x}}\chi(\overline{t},\overline{x},\overline{s},\overline{y})d\overline{t}d\overline{x}d\overline{s}d\overline{y} \\
&\quad+\iiiint\mathrm{sgn}(F(\overline{t}+\overline{s},\overline{x}+\overline{y})-\widetilde{F}(\overline{t}-\overline{s},\overline{x}-\overline{y}))\left(A_{-}(\overline{t}+\overline{s},\overline{t}-\overline{s},F)-A_{-}(\overline{t}+\overline{s},\overline{t}-\overline{s},\widetilde{F})\right)\partial_{\overline{y}}\chi(\overline{t},\overline{x},\overline{s},\overline{y})d\overline{t}d\overline{x}d\overline{s}d\overline{y} \\
&\quad+\iiiint\chi\mathrm{sgn}(F(\overline{t}+\overline{s},\overline{x}+\overline{y})-\widetilde{F}(\overline{t}-\overline{s},\overline{x}-\overline{y}))\left(\mathbf{S}[F](\overline{t}+\overline{s},\overline{x}+\overline{y})-\mathbf{S}[\widetilde{F}](\overline{t}-\overline{s},\overline{x}-\overline{y})\right)d\overline{t}d\overline{x}d\overline{s}d\overline{y}.
\end{aligned}\label{the inequality for I+II+III+IV}
\end{equation}
We handle separately each one of the above integrals. 

\textbf{The integral} $I_{\varepsilon,\delta}$. Letting $\varepsilon\rightarrow0$
we get the 2D integral 

\[
I_{\delta}\coloneqq\underset{\varepsilon\rightarrow0}{\lim}I_{\varepsilon,\delta}=\int\int\left|F(\overline{t},\overline{x})-\widetilde{F}(\overline{t},\overline{x})\right|g\left(\overline{x}\right)h_{\delta}'\left(\overline{t}\right)d\overline{x}d\overline{t}.
\]

We can pass to the limit as $\delta\rightarrow0$ to find that 
\begin{align*}
I_{\delta}&=\int\int\left|F(\overline{t},\overline{x})-\widetilde{F}(\overline{t},\overline{x})\right|g\left(\overline{x}\right)h_{\delta}'\left(\overline{t}\right)\\
&=\frac{1}{\delta}\int_{\sigma-\delta}^{\sigma}\int\left|F(\overline{t},\overline{x})-\widetilde{F}(\overline{t},\overline{x})\right|g\left(\overline{x}\right)d\overline{x}d\overline{t}-\frac{1}{\delta}\int_{\tau}^{\tau+\delta}\int\left|F(\overline{t},\overline{x})-\widetilde{F}(\overline{t},\overline{x})\right|g\left(\overline{x}\right)d\overline{x}d\overline{t}\\ 
&\underset{\delta\rightarrow0}{\rightarrow}\int\left|F(\sigma,\overline{x})-\widetilde{F}(\sigma,\overline{x})\right|g\left(\overline{x}\right)d\overline{x}-\int\left|F(\tau,\overline{x})-\widetilde{F}(\tau,\overline{x})\right|g\left(\overline{x}\right)d\overline{x}\\
&=\int\left|F(\sigma,\overline{x})-\widetilde{F}(\sigma,\overline{x})\right|d\overline{x}-\int\left|F(\tau,\overline{x})-\widetilde{F}(\tau,\overline{x})\right|d\overline{x},   
\end{align*}
\begin{equation}
\label{Inequality for I}
\end{equation}
where in the last equation we relied on the choice of $g$ as well
as fact that $F-\widetilde{F}$ is $0$ outside $[-R,R]$. 

\textbf{The integral} $II_{\varepsilon,\delta}$. Letting $\varepsilon\rightarrow0$
the integral becomes 

\[
II_{\delta}=\underset{\varepsilon\rightarrow0}{\lim}II_{\varepsilon,\delta}=\iint\mathrm{sgn}(F(\overline{t},\overline{x})-\widetilde{F}(\overline{t},\overline{x}))\left(A_{+}(\overline{t},\overline{t},F(\overline{t},\overline{x}))-A_{+}(\overline{t},\overline{t},\widetilde{F}(\overline{t},\overline{x}))\right)g'\left(\overline{x}\right)h_{\delta}\left(\overline{t}\right)d\overline{t}d\overline{x}.
\]
Observe that $g'\left(\overline{x}\right)=0$ for all $\overline{x}\in\mathrm{supp}(A_{+}(\overline{t},\overline{t},F(\overline{t},\cdot))-A_{+}(\overline{t},\overline{t},\widetilde{F}(\overline{t},\cdot)))\subset[-R,R]$,
where the latter inclusion is because of the uniform in time Lipschitz
continuity assumed for $A,\widetilde{A}$. We therefore infer that
the above integral vanishes identically, i.e. 
\begin{equation}
II_{\delta}=0.\label{II}
\end{equation}

\textbf{The integral} $III_{\varepsilon,\delta}$. The following claim
is a straightforward adaptation of Lemma 5.4 in \cite{11}
\begin{claim}
The function $\gamma(\overline{t},F,\widetilde{F})\coloneqq\mathrm{sgn}(F-\widetilde{F})(A_{-}(\overline{t},\overline{t},F)-A_{-}(\overline{t},\overline{t},\widetilde{F}))$
is uniformly in time Lipschitz with respect to the first and second
variable with the bounds 

\[
\left|\gamma(\overline{t},\cdot,\widetilde{F})\right|_{\mathrm{Lip}}\leq\left|A_{-}(\overline{t},\overline{t},\cdot)\right|_{\mathrm{Lip}},\ \left|\gamma(\overline{t},F,\cdot)\right|_{\mathrm{Lip}}\leq\left|A_{-}(\overline{t},\overline{t},\cdot)\right|_{\mathrm{Lip}}.
\]
\end{claim}
\textit{Proof}. Keep $\overline{t}\in[0,T],\widetilde{F}\in[-\frac{1}{2},\frac{1}{2}]$
fixed. Let $F_{1},F_{2}\in[-\frac{1}{2},\frac{1}{2}]$. If $\mathrm{sgn}(F_{1}-\widetilde{F})=\mathrm{sgn}(F_{2}-\widetilde{F})$
then we have 

\[
\left|\gamma(\overline{t},F_{1},\widetilde{F})-\gamma(\overline{t},F_{2},\widetilde{F})\right|=\left|A_{-}(\overline{t},\overline{t},F_{1})-A_{-}(\overline{t},\overline{t},F_{2})\right|\leq\left|A_{-}(\overline{t},\overline{t},\cdot)\right|_{\mathrm{Lip}}\left|F_{1}-F_{2}\right|.
\]
If $\mathrm{sgn}(F_{1}-\widetilde{F})\neq\mathrm{sgn}(F_{2}-\widetilde{F})$
then 
\begin{align*}
 \left|\gamma(\overline{t},F_{1},\widetilde{F})-\gamma(\overline{t},F_{2},\widetilde{F})\right|&=\left|A_{-}(\overline{t},\overline{t},F_{1})-A_{-}(\overline{t},\overline{t},\widetilde{F})+A_{-}(\overline{t},\overline{t},F_{2})-A_{-}(\overline{t},\overline{t},\widetilde{F})\right|\\
 &\leq\left|A_{-}(\overline{t},\overline{t},\cdot)\right|_{\mathrm{Lip}}(\left|F_{1}-\widetilde{F}\right|+\left|F_{2}-\widetilde{F}\right|) \\
 &=\left|A_{-}(\overline{t},\overline{t},\cdot)\right|_{\mathrm{Lip}}\left|F_{1}-\widetilde{F}+\widetilde{F}-F_{2}\right|=\left|A_{-}(\overline{t},\overline{t},\cdot)\right|_{\mathrm{Lip}}\left|F_{1}-F_{2}\right|.   
\end{align*}
\begin{flushright}
$\square$
\par\end{flushright}
We can apply this observation together with Lemma \ref{Chain rule }
in order to integrate by parts and obtain 
\[
III_{\varepsilon,\delta}=-\iiiint\partial_{\overline{y}}\gamma(t,F,\widetilde{F})\chi.
\]
Moreover we get the bound 
\begin{align*}
&\left|\iiiint\partial_{\overline{y}}\gamma(t,F,\widetilde{F})\chi\right| \\
&\quad\leq\underset{\overline{t}\in[0,T]}{\sup}\left|A_{-}(\overline{t},\overline{t},\cdot)\right|_{\mathrm{Lip}}\iiiint\left(\left|\partial_{1}F(\overline{t}+\overline{s},\overline{x}+\overline{y})\right|+\left|\partial_{1}\widetilde{F}(\overline{t}-\overline{s},\overline{x}-\overline{y})\right|\right)b_{\varepsilon}\left(\overline{y}\right)b_{\varepsilon}\left(\overline{s}\right)g\left(\overline{x}\right)h_{\delta}\left(\overline{t}\right).
\end{align*}
Letting $\varepsilon\rightarrow0$ and setting $III_{\delta}\coloneqq\underset{\varepsilon\rightarrow0}{\lim}III_{\varepsilon,\delta}$
we arrive at 

\[
\left|III_{\delta}\right|\leq\underset{\overline{t}\in[0,T]}{\sup}\left|A_{-}(\overline{t},\overline{t},\cdot)\right|_{\mathrm{Lip}}\iint\left(\left|\partial_{1}F(\overline{t},\overline{x})\right|+\left|\partial_{1}\widetilde{F}(\overline{t},\overline{x})\right|\right)g\left(\overline{x}\right)h_{\delta}\left(\overline{t}\right)d\overline{x}d\overline{t}.
\]
Letting $\delta\rightarrow0$ this becomes 
\[
\underset{\delta\rightarrow0}{\lim}\left|III_{\delta}\right|\leq\underset{\overline{t}\in[0,T]}{\sup}\left|A_{-}(\overline{t},\overline{t},\cdot)\right|_{\mathrm{Lip}}\int_{\sigma}^{\tau}\int\left(\left|\partial_{1}F(\overline{t},\overline{x})\right|+\left|\partial_{1}\widetilde{F}(\overline{t},\overline{x})\right|\right)d\overline{x}d\overline{t}.
\]
Since for any fixed $\overline{t}$ 
\[
\int\left(\left|\partial_{1}F(\overline{t},\overline{x})\right|+\left|\partial_{1}\widetilde{F}(\overline{t},\overline{x})\right|\right)=2,
\]
we cocnlude 
\begin{equation}
\underset{\delta\rightarrow0}{\lim}\left|III_{\delta}\right|\leq2\underset{\overline{t}\in[0,T]}{\sup}\left|A_{-}(\overline{t},\overline{t},\cdot)\right|_{\mathrm{Lip}}(\tau-\sigma)=\underset{\overline{t}\in[0,T]}{\sup}\left|A(\overline{t},\cdot)-\widetilde{A}(\overline{t},\cdot)\right|_{\mathrm{Lip}}(\tau-\sigma).\label{Inequality for III}
\end{equation}

\textbf{The integral} $IV_{\varepsilon,\delta}$. The treatment of
this integral reflects the novelty for what concerns the stability
estimate, since the non-local source term in question stands in variance
to the one in \cite{11}. Letting $\varepsilon\rightarrow0$ we get
the 2D integral 

\[
IV_{\delta}\coloneqq\iint\mathrm{sgn}(F(\overline{t},\overline{x})-\widetilde{F}(\overline{t},\overline{x}))\left(\mathbf{S}[F](\overline{t},\overline{x})-\mathbf{S}[\widetilde{F}](\overline{t},\overline{x})\right)g\left(\overline{x}\right)h_{\delta}\left(\overline{t}\right)d\overline{t}d\overline{x}.
\]
The integral $IV_{\delta}$ is now to be bounded by 4 terms, each
of which is mastered separately. 
\begin{align*}
\left|IV_{\delta}\right|&\leq\iint\left|\mathbf{S}[F](t,x)-\mathbf{S}[\widetilde{F}](t,x)\right|g\left(x\right)h_{\delta}\left(t\right)dxdt\\
&\leq\iint\left|F(t,x)\phi\star F(t,x)-\widetilde{F}(t,x)\phi\star\widetilde{F}(t,x)\right|g\left(x\right)h_{\delta}\left(t\right)dxdt\\
&\quad+\iint\left|\int_{-\infty}^{x}\widetilde{F}(t,z)\partial_{z}\phi\star\widetilde{F}(t,z)-F(t,z)\partial_{z}\phi\star F(t,z)dz\right|g\left(x\right)h_{\delta}\left(t\right)dxdt\\
&\leq\iint\left|\left(F(t,x)-\widetilde{F}(t,x)\right)\phi\star F(t,x)\right|g\left(x\right)h_{\delta}\left(t\right)dxdt\\
&\quad+\iint\left|\widetilde{F}(t,x)\phi\star(F-\widetilde{F})(t,x)\right|g\left(x\right)h_{\delta}\left(t\right)dxdt\\
&\quad+\iint\left|\int_{-\infty}^{x}\widetilde{F}(t,z)(\partial_{z}\phi\star(\widetilde{F}-F)(t,z))dz\right|g\left(x\right)h_{\delta}\left(t\right)dxdt\\
&\quad+\iint\left|\int_{-\infty}^{x}\partial_{z}\phi\star F(t,z)(\widetilde{F}-F)(t,z))dz\right|g\left(x\right)h_{\delta}\left(t\right)dxdt\coloneqq\stackrel[k=1]{4}{\sum}J_{k}.
\end{align*}
\textit{Estimate on $J_{1}$}. 
\[
\left|\left(F(t,x)-\widetilde{F}(t,x)\right)\phi\star F(t,x)\right|\leq2R\left\Vert \phi\right\Vert _{\infty}\left|F(t,x)-\widetilde{F}(t,x)\right|,
\]
so that 
\begin{align*}
\left|J_{1}\right|&\leq2R\left\Vert \phi\right\Vert _{\infty}\iint\left|F(t,x)-\widetilde{F}(t,x)\right|g\left(x\right)h_{\delta}\left(t\right)dxdt\\
&\leq2R\left\Vert \phi\right\Vert _{\infty}\int\left\Vert F(t,\cdot)-\widetilde{F}(t,\cdot)\right\Vert _{1}h_{\delta}\left(t\right)dt\underset{\delta\rightarrow0}{\rightarrow}2R\left\Vert \phi\right\Vert _{\infty}\int_{\sigma}^{\tau}\left\Vert F(t,\cdot)-\widetilde{F}(t,\cdot)\right\Vert _{1}dt.
\end{align*}
\textit{Estimate on $J_{2}$}. 
\[
\left|\phi\star(F-\widetilde{F})(t,x)\right|\leq\left\Vert \phi\right\Vert _{\infty}\left\Vert (F-\widetilde{F})(t,\cdot)\right\Vert _{1},
\]
so that 
\begin{align*}
\left|J_{2}\right|&\leq\frac{1}{2}\left\Vert \phi\right\Vert _{\infty}\iint\left\Vert (F-\widetilde{F})(t,\cdot)\right\Vert _{1}g\left(x\right)h_{\delta}\left(t\right)dxdt\\
&\leq2R\left\Vert \phi\right\Vert _{\infty}\int\left\Vert (F-\widetilde{F})(t,\cdot)\right\Vert _{1}h_{\delta}\left(t\right)dt\underset{\delta\rightarrow0}{\rightarrow}2R\left\Vert \phi\right\Vert _{\infty}\int_{\sigma}^{\tau}\left\Vert F(t,\cdot)-\widetilde{F}(t,\cdot)\right\Vert _{1}dt.
\end{align*}
\textit{Estimate on $J_{3}$}. Since $\mathrm{supp}(\partial_{z}\phi\star(\widetilde{F}-F))\subset\mathrm{supp}(\phi)+\mathrm{supp}(\widetilde{F}-F)\subset[-(r+R),r+R]$
it follows that 
\begin{align*}
\left|\int_{-\infty}^{x}\widetilde{F}(t,z)(\partial_{z}\phi\star(\widetilde{F}-F)(t,z))dz\right|&\leq\left|\int_{-(r+R)}^{r+R}\widetilde{F}(t,z)(\partial_{z}\phi\star(\widetilde{F}-F)(t,z))dz\right|\\
&\leq(r+R)\left\Vert \partial_{z}\phi\right\Vert _{\infty}\left\Vert \widetilde{F}(t,\cdot)-F(t,\cdot)\right\Vert _{1}.
\end{align*}
Consequently 
\begin{align*}
\left|J_{3}\right|&\leq(r+R)\left\Vert \partial_{z}\phi\right\Vert _{\infty}\iint\left\Vert \widetilde{F}(t,\cdot)-F(t,\cdot)\right\Vert _{1}g\left(x\right)h_{\delta}\left(t\right)dxdt\\
&\leq(r+R)\left\Vert \partial_{z}\phi\right\Vert _{\infty}\iint\left\Vert \widetilde{F}(t,\cdot)-F(t,\cdot)\right\Vert _{1}g\left(x\right)h_{\delta}\left(t\right)dxdt\\
&\leq4R(r+R)\left\Vert \partial_{z}\phi\right\Vert _{\infty}\int\left\Vert \widetilde{F}(t,\cdot)-F(t,\cdot)\right\Vert _{1}h_{\delta}\left(t\right)dt\\
&\underset{\delta\rightarrow0}{\rightarrow}4R(r+R)\left\Vert \partial_{z}\phi\right\Vert _{\infty}\int_{\sigma}^{\tau}\left\Vert \widetilde{F}(t,\cdot)-F(t,\cdot)\right\Vert _{1}dt.    
\end{align*}
\textit{Estimate on $J_{4}$}. 
\[
\left|\int_{-\infty}^{x}\partial_{z}\phi\star F(t,z)(\widetilde{F}-F)(t,z))dz\right|
\leq2R\left\Vert \partial_{z}\phi\right\Vert _{\infty}\left\Vert \widetilde{F}(t,\cdot)-F(t,\cdot)\right\Vert _{1}.
\]
As a result 
\begin{align*}
\left|J_{4}\right|&\leq2R\left\Vert \partial_{z}\phi\right\Vert _{\infty}\iint\left\Vert \widetilde{F}(t,\cdot)-F(t,\cdot)\right\Vert _{1}g\left(x\right)h_{\delta}\left(t\right)dxdt\\
&\leq8R^{2}\left\Vert \partial_{z}\phi\right\Vert _{\infty}\int\left\Vert \widetilde{F}(t,\cdot)-F(t,\cdot)\right\Vert _{1}h_{\delta}\left(t\right)dt\\
&\underset{\delta\rightarrow0}{\rightarrow}8R^{2}\left\Vert \partial_{z}\phi\right\Vert _{\infty}\int_{\sigma}^{\tau}\left\Vert \widetilde{F}(t,\cdot)-F(t,\cdot)\right\Vert _{1}dt.    
\end{align*}
Summarizing 

\begin{equation}
\left|\underset{\delta\rightarrow0}{\lim}\underset{\varepsilon\rightarrow0}{\lim}IV_{\varepsilon,\delta}\right|\leq C(r,R,\left\Vert \phi\right\Vert _{\infty},\left\Vert \partial_{z}\phi\right\Vert _{\infty})\int_{\sigma}^{\tau}\left\Vert F(t,\cdot)-\widetilde{F}(t,\cdot)\right\Vert _{1}dt.\label{eq:Inequality for IV}
\end{equation}

\textbf{Step 4. Conclusion.} The combination of \eqref{Inequality for I},\eqref{II},\eqref{Inequality for III},
\eqref{eq:Inequality for IV} and \eqref{the inequality for I+II+III+IV}
yields the inequality 
\begin{align*}
\underset{t\in[0,T]}{\sup}\left|A(t,\cdot)-\widetilde{A}(t,\cdot)\right|_{\mathrm{Lip}}(\tau-\sigma)&+C\int_{\sigma}^{\tau}\left\Vert F(t,\cdot)-\widetilde{F}(t,\cdot)\right\Vert _{1}dt\\
&+\int\left|F(\sigma,x)-\widetilde{F}(\sigma,x)\right|dx-\int\left|F(\tau,x)-\widetilde{F}(\tau,x)\right|dx\geq0,
\end{align*}
for some constant $C=C(r,R,\left\Vert \partial_{z}\phi\right\Vert _{\infty},\left\Vert \phi\right\Vert _{\infty})$.
In particular we get 

\[
\left\Vert F(t,\cdot)-\widetilde{F}(t,\cdot)\right\Vert _{1}\leq t\underset{t\in[0,T]}{\sup}\left|A(t,\cdot)-\widetilde{A}(t,\cdot)\right|_{\mathrm{Lip}}+\left\Vert F(0,\cdot)-\widetilde{F}(0,\cdot)\right\Vert _{1}+C\int_{0}^{t}\left\Vert F(s,\cdot)-\widetilde{F}(s,\cdot)\right\Vert _{1}ds,
\]
which by Gronwall implies 

\[
\left\Vert F(t,\cdot)-\widetilde{F}(t,\cdot)\right\Vert _{1}\leq e^{Ct}\left(\left\Vert F(0,\cdot)-\widetilde{F}(0,\cdot)\right\Vert _{1}+t\underset{t\in[0,T]}{\sup}\left|A(t,\cdot)-\widetilde{A}(t,\cdot)\right|_{\mathrm{Lip}}\right),
\]
as desired. 
\begin{flushright}
$\square$
\par\end{flushright}
\begin{rem}
Given initial data $F^{\mathrm{0}}$, it is not difficult to construct
explicitly initial weights $m_{i}^{\mathrm{0}}$ and initial opinions
$x_{i}^{\mathrm{0}}$ such that $\left\Vert -\frac{1}{2}+\frac{1}{N}\stackrel[i=1]{N}{\sum}m_{i}^{\mathrm{0}}H(x-x_{i}^{\mathrm{0}})-F^{\mathrm{0}}\right\Vert _{1} \underset{N\rightarrow\infty}{\rightarrow}0$,
thereby witnessing the fact that the assumption $\left\Vert F_{N}^{\mathrm{0}}-F^{\mathrm{0}}\right\Vert _{1}\underset{N\rightarrow\infty}{\rightarrow}0$
is reasonably typical. See Lemma 5.3 in {[}11{]} for a guidance how
to do this. 
\end{rem}

\section{\label{Appendix A} Appendix A: The Rankine--Hugoniot and Oleinik
Conditions Revisited}

The Rankine--Hugoniot and Oleinik conditions are a very standard tool
in the theory of conservation laws. However it seems that the literature
typically formulates these conditions for conservation laws with time
independent fluxes, no source terms and when there is only a single
curve of discontinuity. All of these additions are completely harmless,
but for the sake of completeness we revisit the derivation of the
entropy inequality subject in this slightly more general settings.
The equation is as usual 
\begin{equation}
\partial_{t}F+\partial_{x}(A(t,F))=\mathbf{S}[F](t,x).\label{appendix burgers}
\end{equation}

\begin{prop}
\textup{(Oleinik condition)} Suppose there are times $0=T_{0}<T_{1}<...<T_{k-1}<T_{k}=T$
such that for each $1\leq j\leq k$ there are curves $\left\{ (t,s_{i}^{j}(t))\right\} _{i=1}^{m_{j}}$
($m_{j}\in\mathbb{N}$) such that $F$ is a classical solution to
Equation (\ref{appendix burgers}) on 
\[
V_{L}^{i,j}\coloneqq\left\{ (t,x)\left|t\in[T_{j-1},T_{j}),s_{i-1}^{j}(t)\leq x<s_{i}^{j}(t)\right.\right\} 
\]
and 
\[
V_{R}^{i,j}\coloneqq\left\{ (t,x)\left|t\in[T_{j-1},T_{j}),s_{i}^{j}(t)<x\leq s_{i+1}^{j}(t)\right.\right\} ,
\]
with the convention $s_{0}^{j}(t)=-\infty$ and $s_{m_{j}+1}^{j}(t)=+\infty$.
For each $t\in[T_{j-1},T_{j})$ let 

\[
F_{L}^{i,j}(t)\coloneqq\underset{x\nearrow s_{i}^{j}(t)}{\lim}F(t,x),\ F_{R}^{i,j}(t)\coloneqq\underset{x\searrow s_{i}^{j}(t)}{\lim}F(t,x).
\]
Suppose that for each $t\in[T_{j-1},T_{j})$ and each $\theta\in(F_{L}^{i,j}(t),F_{R}^{i,j}(t))$
it holds that 
\begin{equation}
\frac{A(t,F_{L}^{i,j}(t))-A(t,F_{R}^{i,j}(t))}{F_{L}^{i,j}(t)-F_{R}^{i,j}(t)}=\dot{s}_{i}^{j}(t)\label{eq:RH}
\end{equation}
and
\begin{equation}
\frac{A(t,\theta)-A(t,F_{L}^{i,j}(t))}{\theta-F_{L}^{i,j}(t)}\geq\dot{s}_{i}^{j}(t).\label{Oleinik condition}
\end{equation}
Then for each $\alpha\in\mathbb{R}$ $F$ satisfies the entropy inequality 
\begin{equation}
\iint_{V}\partial_{t}\chi(F-\alpha)\mathrm{sgn}(F-\alpha)+\partial_{x}\chi\left(A(F)-A(\alpha)\right)\mathrm{sgn}(F-\alpha)+\chi\mathrm{sgn}(F-\alpha)\mathbf{S}[F](t,x)dxdt\geq0.\label{eq:entropy inequality}
\end{equation}
\end{prop}
\textit{Proof}. Let us first consider $\eta:\mathbb{R}\rightarrow\mathbb{R},\psi:\mathbb{R}\rightarrow\mathbb{R}$
where $\eta\in C^{1,1}(\mathbb{R})$ is convex and $\psi'=\eta'A'$.
Then we have that 
\begin{align*}
\iint\partial_{t}\chi\eta(F)+\partial_{x}\chi\psi(F)+\chi\mathbf{S}[F](t,x)dxdt  =&\,\stackrel[j=1]{k}{\sum}\stackrel[i=1]{m_{j}}{\sum}\iint_{V_{L}^{i,j}}\left(\partial_{t}\chi\eta(F)+\partial_{x}\chi\psi(F)+\chi\eta'(F)\mathbf{S}[F](t,x)\right)dxdt \\
&+\stackrel[j=1]{k}{\sum}\stackrel[i=1]{m_{j}}{\sum}\iint_{V_{R}^{i,j}}\left(\partial_{t}\chi\eta(F)+\partial_{x}\chi\psi(F)+\chi\eta'(F)\mathbf{S}[F](t,x)\right)dxdt.
\end{align*}
Set $\Gamma_{i,j}\coloneqq\left\{ (t,s_{i}^{j}(t))\left|t\in[T_{j-1},T_{j})\right.\right\} $.
Keep $1\leq j\leq k$ and $1\leq i\leq m_{j}$ fixed and take a test
function $\chi\in C_{0}^{\infty}((0,T)\times\mathbb{R})$. For readability,
we omit the indices $i,j$. Using that $F$ is a classical solution
on $V_{L},V_{R}$ we get 
\begin{align*}
&\iint_{V_{L}}\left(\partial_{t}\chi\eta(F)+\partial_{x}\chi\psi(F)+\chi\eta'(F)\mathbf{S}[F](t,x)\right)dxdt \\
&\qquad+\iint_{V_{R}}\left(\partial_{t}\chi\eta(F)+\partial_{x}\chi\psi(F)+\chi\eta'(F)\mathbf{S}[F](t,x)\right) dxdt
\\
&\quad=\iint_{V_{L}}\left(-\chi\partial_{t}\eta(F)-\chi\partial_{x}\psi(F)+\chi\eta'(F)\mathbf{S}[F](t,x)\right)dxdt \\
&\qquad+\iint_{V_{R}}\left(-\chi\partial_{t}\eta(F)-\chi\partial_{x}\psi(F)+\chi\eta'(F)\mathbf{S}[F](t,x)\right)dxdt \\
&\qquad+\int_{\Gamma}\left(\chi\eta(F_{L})\nu^{1}+\chi\psi(F_{L})\nu^{2}\right)d\sigma-\int_{\Gamma}\chi\eta(F_{R})\nu^{1}+\chi\psi(F_{R})\nu^{2})d\sigma \\
&\quad=\int_{\Gamma}\left(\chi(\psi(F_{L})-\psi(F_{R}))\nu^{1}+\chi(\eta(F_{L})-\eta(F_{R}))\nu^{2}\right)d\sigma.
\end{align*}
Here 
\[
(\nu^{1},\nu^{2})=\frac{1}{\sqrt{1+\dot{s}^{2}}}(-\dot{s},1),
\]
and the last equality is because of the identity 

\[
\partial_{t}\eta(F)+\partial_{x}\psi(F)+\eta'(F)\mathbf{S}[F]=0,
\]
which is easily derived using that $F$ is a classical solution on
each region. We claim that 

\[
(\eta(F_{L})-\eta(F_{R}))-(\psi(F_{L})-\psi(F_{R}))\dot{s}(t)\geq0.
\]
First, we integrate by parts to find that 
\begin{align*}
\psi(F_{R})-\psi(F_{L})&=\int_{F_{L}}^{F_{R}}\eta'(y)A'(t,y)dy \\
&=-\int_{F_{L}}^{F_{R}}\eta''(y)(A(t,y)-A(t,F_{L}(t)))dy+\left.\eta'(\cdot)(A(t,\cdot)-A(t,F_{L}))\right|_{F_{L}}^{F_{R}}
\\
&=-\int_{F_{L}}^{F_{R}}\eta''(y)(A(t,y)-A(t,F_{L}))dy+\eta'(F_{R}(t))(A(t,F_{R}(t))-A(t,F_{L}(t))),
\end{align*}
and so thanks to the assumption (\ref{Oleinik condition}) and the
convexity of $\eta$ we have the inequality 
\begin{equation}
\begin{aligned}
\psi(F_{L})-\psi(F_{R})&=\int_{F_{L}}^{F_{R}}\eta''(y)(A(t,y)-A(t,F_{L}))dy-\eta'(F_{R})(A(t,F_{R})-A(t,F_{L})) \\
&\geq\dot{s}(t)\int_{F_{L}}^{F_{R}}\eta''(y)(y-F_{L})dy-\eta'(F_{R})(A(t,F_{R})-A(t,F_{L}).\label{eq:-1}
\end{aligned}
\end{equation}
On the other hand 
\begin{align*}
\eta(F_{R})-\eta(F_{L})&=\int_{F_{L}}^{F_{R}}\eta'(y)dy=-\int_{F_{L}}^{F_{R}}\eta''(y)(y-F_{L})dy+\left.(y-F_{L})\eta'\right|_{F_{L}}^{F_{R}} \\
&=-\int_{F_{L}}^{F_{R}}\eta''(y)(y-F_{L}(t))dy+(F_{R}-F_{L})\eta'(F_{R}).
\end{align*}
Together with inequality (\ref{eq:-1}) the last equation entails
\[
\psi(F_{L})-\psi(F_{R})-\dot{s}(t)(\eta(F_{L})-\eta(F_{R}))
\geq\eta'(F_{R})\left(\dot{s}(t)(F_{R}-F_{L})-A(t,F_{R})-A(t,F_{L})\right)=0,
\]
where the last equality is because of the assumption (\ref{eq:RH}).
It follows that 
\begin{equation}
\int\partial_{t}\chi\eta(F)+\partial_{x}\chi\psi(F)+\chi\eta'(F)\mathbf{S}[F](t,x)dxdt\geq0.\label{eq:entropy inequality for regular flux pair}
\end{equation}
To finish, we use a standard approximation argument. For each $\varepsilon>0$
consider the convex function $\mathfrak{s}_{\varepsilon}\in C^{1,1}(\mathbb{R})$
defined by 
\[
\mathfrak{s}_{\varepsilon}(z)\coloneqq
\begin{cases}
\frac{1}{2\varepsilon}z^{2}, & \left|z\right|\leq\varepsilon, \\
\left|z\right|-\frac{\varepsilon}{2}, & \left|z\right|>\varepsilon.
\end{cases}
\]
and for each $\alpha\in\mathbb{R}$ let 
\[
\eta_{\varepsilon}(z)\coloneqq\mathfrak{s}_{\varepsilon}(z-\alpha).
\]
It is clear that we have pointwise convergence 

\[
\eta_{\varepsilon}\underset{\varepsilon\rightarrow0}{\rightarrow}\left|z-\alpha\right|,\ \eta_{\varepsilon}'\underset{\varepsilon\rightarrow0}{\rightarrow}\mathrm{sgn}(z-\alpha).
\]
Furthermore we can take 

\[
\psi_{\varepsilon}(F)=\int_{\alpha}^{F}\eta_{\varepsilon}'A'(t,y)dy=\int_{\alpha}^{F}\eta_{\varepsilon}'(A(t,y)-A(t,\alpha))'dy,
\]
and as $\varepsilon\rightarrow0$ we get 

\[
\psi_{\varepsilon}(F)\underset{\varepsilon\rightarrow0}{\rightarrow}\int_{\alpha}^{F}\mathrm{sgn}(y-\alpha)(A(t,y)-A(t,\alpha))'dy=\mathrm{sgn}(F-\alpha)(A(t,F)-A(t,\alpha))
\]
pointwise. Therefore testing inequality (\ref{eq:entropy inequality for regular flux pair})
with $\eta_{\varepsilon},\psi_{\varepsilon}$ and taking the limit
as $\varepsilon\rightarrow0$ yields the asserted inequality (\ref{eq:entropy inequality}). 
\begin{flushright}
$\square$
\par\end{flushright}

\subsubsection*{Acknowledgments}
IBP and JAC were supported by the EPSRC grant number EP/V051121/1.
This work was also supported by the Advanced Grant Nonlocal-CPD (Nonlocal PDEs for Complex Particle Dynamics: Phase Transitions, Patterns and Synchronization) of the European Research Council Executive Agency (ERC) under the European Union's Horizon 2020 research and innovation programme (grant agreement No. 883363). STG was supported by the Research Council of Norway through grant no.~286822, \textit{Wave Phenomena and Stability --- a Shocking Combination (WaPheS)}. We thank the comments of the anonymous referee which have improved the quality of this work.


\end{document}